\documentclass{amsart}

\usepackage[english]{babel}
\usepackage{amsmath, amsfonts, amssymb, amsthm, faktor}

\usepackage[all]{xy}
\usepackage{tikz}
\usetikzlibrary{math}
\usetikzlibrary{patterns}
\usepackage{caption}
\usepackage{subfig}
\usepackage{float}

\usepackage{caption}
\usetikzlibrary{arrows,chains,matrix,positioning,scopes,decorations.pathreplacing,
decorations.pathmorphing,decorations.markings,arrows.meta}

\usepackage{aliascnt}
\usepackage[colorlinks, linktocpage, allcolors=black,breaklinks]{hyperref}

\usepackage{enumerate}

\usepackage{verbatim}
\usepackage[T1]{fontenc}
\usepackage{lmodern}
\usepackage{amsmath,amssymb,amsthm,mathtools}
\usepackage{booktabs,longtable,array,tabularx}
\usepackage{enumitem}
\usepackage{fancyvrb}
\usepackage{multicol}
\usepackage[margin=1in]{geometry}
\usepackage{microtype}
\allowdisplaybreaks

\newenvironment{appcerttable}{%
  \par\addvspace{0.55\baselineskip}%
  \noindent\begin{minipage}{\linewidth}%
  \centering\footnotesize%
  \captionsetup{font=footnotesize,skip=3pt}%
  \setlength{\tabcolsep}{2.0pt}%
}{%
  \end{minipage}\par\addvspace{0.55\baselineskip}%
}

\theoremstyle{definition}
\newtheorem{theorem}{Theorem}[section]

\newtheorem{proposition}[theorem]{Proposition}
\newtheorem{corollary}[theorem]{Corollary}
\newtheorem{lemma}[theorem]{Lemma}

\newtheorem*{remark}{Remark}

\newtheorem*{acknowledgments}{Acknowledgments}

\newcommand{\Tcqadpa}{T_{c^qa=ad^p}}

\newcommand{\Gcaca}{G_{ca=ac}}
\newcommand{\Gdada}{G_{da=ad}}
\newcommand{\Gcbcb}{G_{cb=bc}}
\newcommand{\Gdbdb}{G_{db=bd}}
\newcommand{\Gdcadca}{G_{dca=acd}}
\newcommand{\Gcdacda}{G_{cda=adc}}
\newcommand{\Gcbacba}{G_{cba=abc}}
\newcommand{\Gcabcab}{G_{cab=bac}}
\newcommand{\Gcqadpa}{G_{c^qa=ad^p}}
\newcommand{\Gdpacqa}{G_{d^pa=ac^q}}
\newcommand{\Gcbpcaq}{G_{cb^p=a^qc}}
\newcommand{\Gcaqcbp}{G_{ca^q=b^pc}}
\newcommand{\Z}{\mathbb Z}
\newcommand{\vecF}{\vec F(2^{\Z^2})}

\newcommand{\Shift}{\vec F(2^{\Z^2})}
\newcommand{\GammaT}{\vec\Gamma}
\newcommand{\chic}{\chi_{co}}
\newcommand{\Cp}{\vec C}
\newcommand{\PathTwo}{\mathcal P_2}
\tikzset{
  BdyNode/.style         ={circle, fill, inner sep=0pt, minimum size=3pt},
  nodeStyle/.style       ={circle, draw, fill=white, inner sep=0pt, minimum size=7pt},
  edgeStrongStyle/.style ={line width=3pt, line cap=rect},
  BoxLabelStyle/.style   ={scale=1.2},
  BoxStyle/.style        ={scale=1, anchor=north west, rectangle, draw},
  XBoxStyle/.style       ={BoxStyle, minimum width=1cm,   minimum height=1cm},
  ABoxStyle/.style       ={BoxStyle, minimum width=1cm,   minimum height=0.8cm},
  BBoxStyle/.style       ={BoxStyle, minimum width=1cm,   minimum height=1.2cm},
  CBoxStyle/.style       ={BoxStyle, minimum width=0.8cm, minimum height=1cm},
  DBoxStyle/.style       ={BoxStyle, minimum width=1.2cm, minimum height=1cm}
}

\newcommand{\cT}{\mathcal T}

\newcommand{\Pfive}{P_5}
\begin{document}
\title[Continuous oriented chromatic number]
{The continuous oriented chromatic number of directed Schreier graphs of $\mathbb Z^2$-shift actions}

\author[Ruijun Wang]{Ruijun Wang}
\address{School of Mathematical Sciences and School of Pre-university, Dalian Minzu University}
\email{wangruijun@dlnu.edu.cn}
\thanks{}



\date{}

\maketitle
\begin{abstract}
Let $\vec F(2^{\mathbb Z^2})$ be the directed Schreier graph on the free part of the Bernoulli shift $\mathbb Z^2\curvearrowright 2^{\mathbb Z^2}$, with arcs in the two coordinate directions. We prove that the continuous oriented chromatic number of it is 7, that is, there is a tournament on 7 vertices receiving a continuous graph homomorphism from $\vec F(2^{\mathbb Z^2})$ and there is no continuous graph homomorphism from $\vec F(2^{\mathbb Z^2})$ to any tournament on 6 vertices. And we prove that the Borel and measurable oriented chromatic number of directed Schreier graph $\vec F(2^{\mathbb Z^n})$, $n>1$ is 5.
\end{abstract}

\maketitle
\tableofcontents

\section{Introduction}

In the seminal paper \cite{KST1999}, Kechris, Solecki and Todorcevic initiated the fundamental descriptive combinatorics theory. In particular, they proved that a Borel graph of bounded degree $d$ has Borel chromatic number at most $d+1$. In \cite{Marks}, Marks proved that the Borel chromatic number of the Schreier graph $F(2^{\mathbb F_2})$ is $2n+1$ which reaches the upper bound. In \cite{GJ2015}, Gao and Jackson developed the rectangular partition method for Schreier graphs of countable abelian groups actions, and they proved that there is a continuous proper 4-coloring for the Schreier graph $F(2^{\mathbb Z^2})$. Subsequently, in \cite{GJKS2025}, Gao, Jackson, Krohne and Seward studied the hyper-aperiodic elements method and proved the Twelve Tiles Theorem, which converts many continuous questions on \(F(2^{\mathbb Z^2})\) into finite questions on twelve rectangular tiles, and they prove that there is no continuous proper 3-coloring for the Schreier graph $F(2^{\mathbb Z^2})$. Combined with previous results, we know that the continuous chromatic number is 4.

An oriented coloring $c$ of an oriented graph $G$ is a graph homomorphism from $G$ to an oriented graph $H$. The oriented chromatic number $\chi_o(G)$ of an oriented graph $G$ is the least size of an oriented graph $H$ to which there is a graph homomorphism.$$\chi_o(G)=\inf\{|V(H)|:c:G\to H\text{ is a graph homomorphism, }H\text{ is an oriented graph}\}.$$
It resembles proper colorings and the chromatic number, see Kostochka, Sopena and Zhu \cite{KSZ1997} and Sopena \cite{Sopena1997}. For a topological oriented graph, we can define its continuous or Borel oriented chromatic number by requiring the coloring to be continuous or Borel respectively.

Our main result is the following.

\begin{theorem}\label{thm:main}
The continuous oriented chromatic number of the directed Schreier graph
$\vecF$ is
\[
             \chic(\Shift)=7.
\]
\end{theorem}

For the construction part, we show that there is a graph homomorphism from $\GammaT_{1,3,4}$ to a tournament on 7 vertices. For the obstruction part, we show that there is no graph homomorphism from $\GammaT_{n,p,q}$ to any tournament on 6 vertices, the proof uses only two of the twelve tiles.

\begin{enumerate}[label=\textup{(\roman*)}]
\item A \emph{long tile} has a top boundary consisting of \(q\) copies
of a directed \(p\)-cycle and a bottom boundary consisting of \(p\)
copies of a directed \(q\)-cycle.  A curl-free edge energy equates the
two boundary energies.  Since \(\gcd(p,q)=1\), an energy forces
both boundary cycles to be energy-monochromatic.

\item A \emph{torus tile} is
\(\Cp_p\square\Cp_q\).  For the four tournaments not handled by an
exact edge energy, a local marker attached to directed triples propagates
diagonally across this torus.  The marker set is therefore invariant
under translation by \(q\) in \(\mathbb Z/p\mathbb Z\), which is
impossible when it is nonempty and proper.
\end{enumerate}

The classical classification says that there
are 35 non-isomorphic strong tournaments on six vertices, see,
for example, Moon's monograph \cite{Moon1968}. 27 of them are handled uniformly by an
order energy method, 4 more by non-order energies, and the
last four by two marker systems.

In \cite{GJKS}, Gao, Jackson, Krohne and Seward developed the Borel toast structure and proved that the Borel chromatic number $\chi_B(F(2^{\mathbb Z^n}))=3$. In \cite{CJMST}, Conley, Jackson, Marks, Seward and Tucker-Drob showed an alternative proof using Borel asymptotic dimension. In this paper, we show the following theorem.

\begin{theorem}\label{thm:Borel}
    The Borel oriented chromatic number of the directed Schreier graph
$\vec F(2^{\mathbb Z^n})$, $n>1$ is
\[
             \chi_{Bo}(\vec F(2^{\mathbb Z^n}))=5.
\]
\end{theorem}

For the construction part, we use the Borel toast structure and show that there is a Borel graph homomorphism from $\vec F(2^{\mathbb Z^n})$, $n>1$ to the regular tournament on 5 vertices. For the obstruction part, we show that there is no Borel (in fact measurable) graph homomorphism from $\vec F(2^{\mathbb Z^n})$, $n>1$ to the strong tournament on 4 vertices.

This paper is partially generated by artificial intelligence. The AI model is gpt 5.5. The author had some partial results but was stuck with computer coding. The AI serves as a programmer based on human ideas. All AI generated contents are verified by human.

The rest of the paper is organized as follows. In Section 2, we fix some notations and introduce the directed twelve tiles theorem, we fix an enumeration of 35 non-isomorphic strong tournaments on 6 vertices. In Section 3, we show that by the energy function method, there is no graph homomorphism from the long tile to 31 different tournaments. In Section 4, we show that by marker of directed path with length 2, there is no graph homomorphism from the torus tile to the remaining tournaments. In Section 5, we show that there is a continuous oriented coloring with 7 colors. In Section 6, we show that the Borel oriented chromatic number $\chi_{Bo}(\vec F(2^{\mathbb Z^n}) )=5$, $n>1$. In Section 7, we show a proposition similar to \cite[Theorem 4.4.1]{GJKS2025}. In Section 8, we discuss some future work. We leave some certificates in Appendix A and Appendix B.

\section{Preliminaries}\label{sec:prelim}

\subsection{Continuous oriented chromatic number}

A \emph{directed graph} $D$ or \emph{digraph} is an irreflexive binary relation $A(D)$ on a vertex set $V(D)$, $(x,x)\notin A(D)$ where $A(D)$ is the arc set. An \emph{oriented graph} $O$ is a digraph and an asymmetric binary relation $A(O)$, $(x,y)\in A(O)\Rightarrow (y,x)\notin A(O)$. An \emph{oriented coloring} $c$ is a graph homomorphism \(c:G\to H\) maps every arc of \(G\)
to an arc of \(H\) where both $G$ and $H$ are oriented graphs. The oriented chromatic number is
\[
 \chi_o(G)=\min\{|V(H)|:H\text{ is oriented and }c:G\to H\text{ is a graph homomorphism}\}.
\]
For example, the oriented chromatic number of the directed Cayley graph of $\mathbb Z^n$ with standard generators is 3.

For a topological graph, \emph{the continuous oriented chromatic number} \(\chic\) is defined by requiring the graph
homomorphism to be continuous.

Every oriented graph on \(k\) vertices is a spanning subdigraph
of a tournament on \(k\) vertices. Consequently, without loss of generality, we can assume that $H$ is a tournament.

\subsection{35 non-isomorphic strong tournaments}

A \emph{tournament} $T$ is a complete oriented graph, i.e., for any unordered pair $\{u,v\}$ either $(u,v)\in A(T)$ or $(v,u)\in A(T)$. A tournament is called a \emph{strong tournament} if it is strongly connected. There are 56 different non-isomorphic tournaments on 6 vertices, and 35 of them are strong, see, for example, \cite{Tournament, OEIS-A051337}. A tournament is determined by its set \(B\) of \emph{backward arcs}
\(j\to i\) with \(j>i\), and every pair not in \(B\) is oriented forward.

We define the \emph{canonical code} of a tournament $T$. Let $p_0=01,p_1=02,p_2=03,\cdots,p_{\binom{n}{2}-1}=(n-2)(n-1)$ be pairs in lexicographic order. For any re-assignment of labeling $\sigma\in S_n$, let
\[
 b^\sigma_k=
 \begin{cases}
 1,&\sigma(i_k)\to \sigma(j_k),\\
 0,&\sigma(j_k)\to \sigma(i_k).
 \end{cases}
\]
Then the code of labeling $\sigma$ is $$c_\sigma(T)=\sum_0^{\binom{n}{2}-1}b^\sigma_k2^k.$$ The canonical code of tournament $T$ is $c(T)=\min\limits_{\sigma\in S_n}c_\sigma(T)$. Two non-isomorphic tournaments have different canonical codes.

We illustrate all 35 strong tournaments as follows. The score sequence of a tournament is a sequence of out-degrees. The codes are pairwise different and therefore the list is exhaustive.

\begin{center}
\small
\begin{longtable}{@{}c c >{\raggedright\arraybackslash}p{0.44\textwidth} c@{}}
\toprule
case&score sequence&backward arcs \(B\)&code\\
\midrule
\endfirsthead
\toprule
case&score sequence&backward arcs \(B\)&code\\
\midrule
\endhead
$T_{1}$ & $(1,1,2,3,4,4)$ & $\{50\}$ & 16 \\
$T_{2}$ & $(1,2,2,2,4,4)$ & $\{50,52\}$ & 20 \\
$T_{3}$ & $(1,2,2,3,3,4)$ & $\{50,51\}$ & 24 \\
$T_{4}$ & $(1,1,3,3,3,4)$ & $\{50,51,53\}$ & 26 \\
$T_{5}$ & $(1,2,2,3,3,4)$ & $\{50,51,52\}$ & 28 \\
$T_{6}$ & $(1,2,2,2,4,4)$ & $\{21,41,50,51\}$ & 81 \\
$T_{7}$ & $(1,1,2,3,4,4)$ & $\{21,30,40,41,43,51\}$ & 83 \\
$T_{8}$ & $(1,2,2,3,3,4)$ & $\{41,50\}$ & 144 \\
$T_{9}$ & $(1,2,2,3,3,4)$ & $\{21,30,40,41,50,51\}$ & 145 \\
$T_{10}$ & $(1,2,2,3,3,4)$ & $\{21,30,40,41,50,51,53\}$ & 146 \\
$T_{11}$ & $(1,1,3,3,3,4)$ & $\{21,40,41,50,51,53\}$ & 147 \\
$T_{12}$ & $(2,2,2,2,3,4)$ & $\{41,50,52\}$ & 148 \\
$T_{13}$ & $(1,2,2,3,3,4)$ & $\{21,30,41,50,51\}$ & 149 \\
$T_{14}$ & $(2,2,2,2,3,4)$ & $\{41,50,51\}$ & 152 \\
$T_{15}$ & $(1,2,2,3,3,4)$ & $\{10,32,40,50,51,54\}$ & 153 \\
$T_{16}$ & $(1,2,2,3,3,4)$ & $\{41,50,51,53\}$ & 154 \\
$T_{17}$ & $(2,2,2,2,3,4)$ & $\{41,50,51,52\}$ & 156 \\
$T_{18}$ & $(1,2,2,3,3,4)$ & $\{31,50,51\}$ & 177 \\
$T_{19}$ & $(1,2,2,3,3,4)$ & $\{31,50,51,52\}$ & 181 \\
$T_{20}$ & $(2,2,2,3,3,3)$ & $\{40,50,51\}$ & 280 \\
$T_{21}$ & $(1,2,3,3,3,3)$ & $\{31,40,41,50,51\}$ & 281 \\
$T_{22}$ & $(1,2,3,3,3,3)$ & $\{40,50,51,53\}$ & 282 \\
$T_{23}$ & $(1,2,3,3,3,3)$ & $\{30,50,51,53\}$ & 313 \\
$T_{24}$ & $(2,2,2,3,3,3)$ & $\{40,42,50,51\}$ & 344 \\
$T_{25}$ & $(2,2,2,3,3,3)$ & $\{31,40,50,51\}$ & 345 \\
$T_{26}$ & $(2,2,2,3,3,3)$ & $\{40,41,50,51\}$ & 408 \\
$T_{27}$ & $(2,2,2,3,3,3)$ & $\{31,40,42,50,51\}$ & 1332 \\
$T_{28}$ & $(1,2,2,2,4,4)$ & $\{42,50\}$ & 80 \\
$T_{29}$ & $(2,2,2,2,3,4)$ & $\{42,50,51\}$ & 88 \\
$T_{30}$ & $(1,1,3,3,3,4)$ & $\{31,50\}$ & 176 \\
$T_{31}$ & $(1,2,3,3,3,3)$ & $\{31,40,50\}$ & 377 \\
$T_{32}$ & $(1,1,2,3,4,4)$ & $\{10,21,30,31,32,40,41,42,43,51,52,53,54\}$ & 18 \\
$T_{33}$ & $(1,2,2,3,3,4)$ & $\{20,21,30,32,41,42,43,51,52,53,54\}$ & 89 \\
$T_{34}$ & $(1,1,2,3,4,4)$ & $\{20,50\}$ & 22 \\
$T_{35}$ & $(1,2,2,3,3,4)$ & $\{10,31,40,50\}$ & 150 \\
\bottomrule
\end{longtable}
\end{center}

\subsection{The directed Schreier graph}

A directed Schreier graph of a marked group is a natural oriented graph if there is no generator of index 2. Let \(2^{\mathbb Z^2}\) equip with the product topology which is homeomorphic to the Cantor space $\{0,1\}^\mathbb N$ and carry the right shift action
\[
       (g\cdot x)(h)=x(h+g)
       \qquad(g,h\in\mathbb Z^2).
\]
And its free part is
\[
 F(2^{\mathbb Z^2})
   =\{x:g\cdot x\ne x\text{ for every }(0,0)\ne g\in\mathbb Z^2\}
\]
which is a $G_\delta$ subspace, and hence it is a Polish space.

Let $\{$\(e_1=(1,0)\), \(e_2=(0,1)\)$\}$ be the set of standard generators. The \emph{directed Schreier graph}
\(\Shift\) has vertex set \(F(2^{\mathbb Z^2})\) and arcs
\[
             x\longrightarrow e_1\cdot x,
             \qquad
             x\longrightarrow e_2\cdot x.
\]
Every connected component of it is a copy of the \emph{directed Cayley graph} \(\mathbb Z^2\) with standard generators, and it is an oriented graph.

\subsection{The Twelve Tiles Theorem}

For integers $1\le n<p,q$ with $\gcd(p,q)=1$,
Gao, Jackson, Krohne, and Seward construct a finite graph
\(\Gamma_{n,p,q}\) from twelve rectangular tiles by identifying equally
labeled boundary blocks.  Orient every horizontal and vertical edge in
the positive coordinate direction and write the resulting digraph as
\(\GammaT_{n,p,q}\). The following is the specialization of the Twelve Tiles Theorem
\cite[Theorem 2.3.5]{GJKS2025}.

\begin{theorem}[Directed Twelve Tiles Theorem]\label{thm:twelve}
For a finite directed graph \(H\), the following are equivalent.
\begin{enumerate}[label=\textup{(\roman*)}]
\item There is a continuous graph homomorphism \(\Shift\to H\).
\item For some \(1\le n<p,q\) with \(\gcd(p,q)=1\), there is a graph homomorphism \(\GammaT_{n,p,q}\to H\).
\item For every \(n\), and all sufficiently large coprime \(p,q>n\),
there is a graph homomorphism \(\GammaT_{n,p,q}\to H\).
\end{enumerate}
\end{theorem}

In the figures, five blocks $R_\times,R_a,R_b,R_c,R_d$ are rectangular oriented grid graphs. The sizes of them are as follows.

\begin{alignat*}{8}
R_\times: &\quad\quad &  n   &  &\;\times\;&&      &n  \\
R_a: &&             n& &\;\times\;&&      (p-&n)  \\
R_b: &&             n& &\;\times\;&&      (q-&n)  \\
R_c: &&               (p-n&)  &\;\times\;&& &n \\
R_d: &&              (q- n &)  &\;\times\;&& &n
\end{alignat*}

Each directed tile is a rectangular oriented grid graph with labeled blocks. $\GammaT_{n,p,q}$ is the quotient graph of the twelve directed tiles by identifying equally
labeled boundary blocks.

The following two geometric facts of two of the tiles will be used in the proof.

\begin{enumerate}[label=\textup{(\alph*)}]
\item The tile denoted \(T_{c^q a=ad^p}\) has top boundary equal to
\(q\) repetitions of the same directed closed walk of length \(p\),
bottom boundary equal to \(p\) repetitions of the same directed closed
walk of length \(q\), and identical left and right boundary color
sequences, see Figure \ref{fig:Long}.

\item The tile denoted \(T_{cb=bc}\), after its opposite boundary
blocks are identified, is the \emph{directed Cartesian torus}
\(
                        \Cp_p\square\Cp_q
\) where the \emph{Cartesian product} $D\square D'$ of two digraphs $D$ and $D'$ is the digraph with vertex set $V(D)\times V(D')$ and arc set consists of arcs $(i,j)\to (i',j')$ if
$i=i',\ j\to j'\in A(D')\quad\text{ or }\quad i\to i'\in A(D),\ j=j'.$
\end{enumerate}

If \(p=2\) or \(q=2\), the directed graph $\vec \Gamma_{n,p,q}$ contains a directed two-cycle and thus
cannot map to a tournament. Thus, we may
assume \(p,q\ge3\) when we are focused on oriented chromatic number.


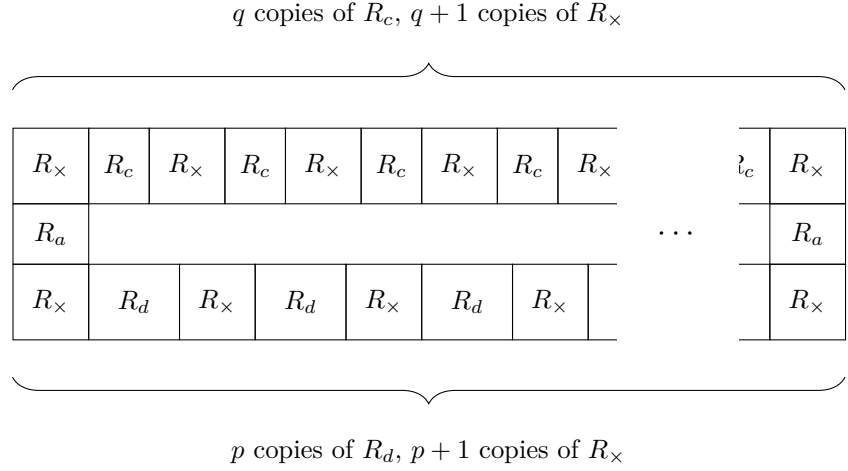
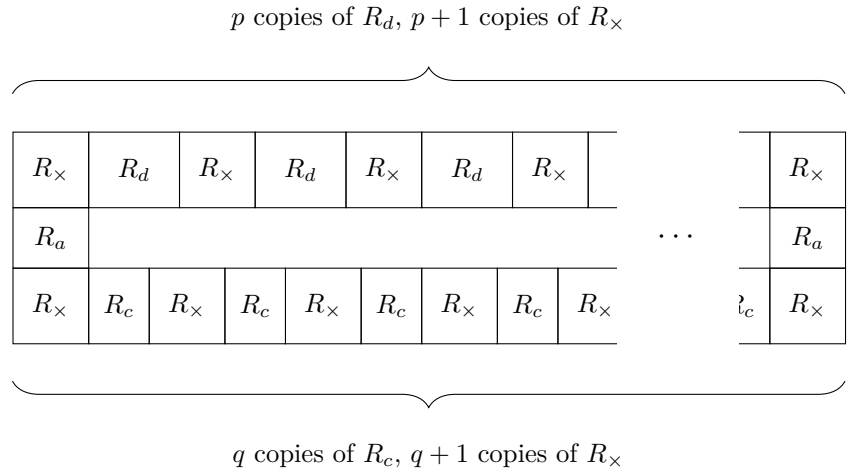
\begin{figure}[H] 
\begin{tikzpicture}
\begin{scope}[scale=1,yscale=-1]
\begin{scope}[shift={(-5,2)}]  
  \foreach \loc in {(0,0),(0,1.8),(1.8,0),(1.8,1.8)} {
    \node[XBoxStyle] at \loc {$R_\times$};
  }
  \foreach \loc in {(0,1),(1.8,1)} {
    \node[ABoxStyle] at \loc {$R_a$};
  }
  \foreach \loc in {(1,0),(1,1.8)} {
    \node[CBoxStyle] at \loc {$R_c$};
  }

  \node[BoxLabelStyle] at (1.4,3.4) {\small $G^1_{n,p,q}$ or $\Gcaca$};
\end{scope}

\begin{scope}[shift={(-5,7)}]  
  \foreach \loc in {(0,0),(0,1.8),(2.2,0),(2.2,1.8)} {
    \node[XBoxStyle] at \loc {$R_\times$};
  }
  \foreach \loc in {(0,1),(2.2,1)} {
    \node[ABoxStyle] at \loc {$R_a$};
  }
  \foreach \loc in {(1,0),(1,1.8)} {
    \node[DBoxStyle] at \loc {$R_d$};
  }

  \node[BoxLabelStyle] at (1.6,3.4) {\small $G^3_{n,p,q}$ or $\Gdada$};
\end{scope}

\begin{scope}[shift={(0,2)}]  
  \foreach \loc in {(0,0),(0,2.2),(1.8,0),(1.8,2.2)} {
    \node[XBoxStyle] at \loc {$R_\times$};
  }
  \foreach \loc in {(0,1),(1.8,1)} {
    \node[BBoxStyle] at \loc {$R_b$};
  }
  \foreach \loc in {(1,0),(1,2.2)} {
    \node[CBoxStyle] at \loc {$R_c$};
  }

  \node[BoxLabelStyle] at (1.4,3.8) {\small $G^2_{n,p,q}$ or $\Gcbcb$};
\end{scope}

\begin{scope}[shift={(0,7)}]  
  \foreach \loc in {(0,0),(0,2.2),(2.2,0),(2.2,2.2)} {
    \node[XBoxStyle] at \loc {$R_\times$};
  }
  \foreach \loc in {(0,1),(2.2,1)} {
    \node[BBoxStyle] at \loc {$R_b$};
  }
  \foreach \loc in {(1,0),(1,2.2)} {
    \node[DBoxStyle] at \loc {$R_d$};
  }

  \node[BoxLabelStyle] at (1.6,3.8) {\small $G^4_{n,p,q}$ or $\Gdbdb$};
\end{scope}

\end{scope}
\end{tikzpicture}
\caption{\label{fig:Gamma-npq-torus}The torus tiles in $\Gamma_{n,p,q}$. }
\end{figure} 

\begin{figure}[H] 
\begin{tikzpicture}
\begin{scope}[scale=1,yscale=-1]
\begin{scope}[shift={(5,2)}]  
  \foreach \loc in {(0,0),(2.2,0),(4,0),(0,1.8),(1.8,1.8),(4,1.8)} {
    \node[XBoxStyle] at \loc {$R_\times$};
  }
  \foreach \loc in {(0,1),(4,1)} {
    \node[ABoxStyle] at \loc {$R_a$};
  }
  \foreach \loc in {(1,0),(2.8,1.8)} {
    \node[DBoxStyle] at \loc {$R_d$};
  }
  \foreach \loc in {(1,1.8),(3.2,0)} {
    \node[CBoxStyle] at \loc {$R_c$};
  }

  \node[BoxLabelStyle] at (2.5,3.4) {\small $G^5_{n,p,q}$ or $\Gdcadca$};
\end{scope}

\begin{scope}[shift={(5,10)}]  
  \foreach \loc in {(0,0),(1.8,0),(4,0),(0,1.8),(2.2,1.8),(4,1.8)} {
    \node[XBoxStyle] at \loc {$R_\times$};
  }
  \foreach \loc in {(0,1),(4,1)} {
    \node[ABoxStyle] at \loc {$R_a$};
  }
  \foreach \loc in {(1,0),(3.2,1.8)} {
    \node[CBoxStyle] at \loc {$R_c$};
  }
  \foreach \loc in {(1,1.8),(2.8,0)} {
    \node[DBoxStyle] at \loc {$R_d$};
  }

  \node[BoxLabelStyle] at (2.5,3.4) {\small $G^7_{n,p,q}$ or $\Gcdacda$};
\end{scope}

\begin{scope}[shift={(12,1)}]  
  \foreach \loc in {(0,0),(0,1.8),(0,4),(1.8,0),(1.8,2.2),(1.8,4)} {
    \node[XBoxStyle] at \loc {$R_\times$};
  }
  \foreach \loc in {(1,0),(1,4)} {
    \node[CBoxStyle] at \loc {$R_c$};
  }
  \foreach \loc in {(0,1),(1.8,3.2)} {
    \node[ABoxStyle] at \loc {$R_a$};
  }
  \foreach \loc in {(1.8,1),(0,2.8)} {
    \node[BBoxStyle] at \loc {$R_b$};
  }

  \node[BoxLabelStyle] at (1.4,5.6) {\small $G^6_{n,p,q}$ or $\Gcbacba$};
\end{scope}

\begin{scope}[shift={(12,9)}]  
  \foreach \loc in {(0,0),(0,2.2),(0,4),(1.8,0),(1.8,1.8),(1.8,4)} {
    \node[XBoxStyle] at \loc {$R_\times$};
  }
  \foreach \loc in {(1,0),(1,4)} {
    \node[CBoxStyle] at \loc {$R_c$};
  }
  \foreach \loc in {(0,1),(1.8,2.8)} {
    \node[BBoxStyle] at \loc {$R_b$};
  }
  \foreach \loc in {(1.8,1),(0,3.2)} {
    \node[ABoxStyle] at \loc {$R_a$};
  }

  \node[BoxLabelStyle] at (1.4,5.6) {\small $G^8_{n,p,q}$ or $\Gcabcab$};
\end{scope}
\end{scope}
\end{tikzpicture}
\caption{\label{fig:Gamma-npq-commute}The commutativity tiles in $\Gamma_{n,p,q}$.}
\end{figure} 

\begin{figure}[H] 
\begin{tikzpicture}
\begin{scope}[scale=1,yscale=-1]
\begin{scope}[shift={(0,0)}]  
  \foreach \loc in {(0,0.0),(1.8,0.0),(3.6,0.0),(5.4,0.0),(7.2,0.0),(10,0.0),
                    (0,1.8),(2.2,1.8),(4.4,1.8),(6.6,1.8),(10,1.8)} {
    \node[XBoxStyle] at \loc {$R_\times$};
  }
  \foreach \loc in {(0,1),(10,1)} {
    \node[ABoxStyle] at \loc {$R_a$};
  }
  \foreach \loc in {(1,0),(2.8,0),(4.6,0),(6.4,0),(9.2,0)} {
    \node[CBoxStyle] at \loc {$R_c$};
  }
  \foreach \loc in {(1,1.8),(3.2,1.8),(5.4,1.8),(7.6,1.8),(8.8,1.8)} {
    \node[DBoxStyle] at \loc {$R_d$};
  }

  \fill[fill=white] (8,-.1) rectangle (9.6, 2.9);

  \node at (8.8, 1.4) {\Large$\cdots$};

  \draw [decorate,decoration={brace,amplitude=10pt}]
  (0,-0.5) -- (11,-0.5) node [scale=1, black,midway, above, yshift=20pt]
        {$q$ copies of $R_c$, $q+1$ copies of $R_\times$};

  \draw [decorate,decoration={brace,amplitude=10pt}]
  (11,3.3) -- (0,3.3) node [scale=1, black,midway, below, yshift=-20pt]
        {$p$ copies of $R_d$, $p+1$ copies of $R_\times$};

  \node[BoxLabelStyle] at (5.5,5.5) {\small $G^9_{n,p,q}$ or $\Gcqadpa$};
\end{scope}

\begin{scope}[shift={(0,10)}]  
  \foreach \loc in {(0,1.8),(1.8,1.8),(3.6,1.8),(5.4,1.8),(7.2,1.8),(10,1.8),
                    (0,0),(2.2,0),(4.4,0),(6.6,0),(10,0)} {
    \node[XBoxStyle] at \loc {$R_\times$};
  }
  \foreach \loc in {(0,1),(10,1)} {
    \node[ABoxStyle] at \loc {$R_a$};
  }
  \foreach \loc in {(1,1.8),(2.8,1.8),(4.6,1.8),(6.4,1.8),(9.2,1.8)} {
    \node[CBoxStyle] at \loc {$R_c$};
  }
  \foreach \loc in {(1,0),(3.2,0),(5.4,0),(7.6,0),(8.8,0)} {
    \node[DBoxStyle] at \loc {$R_d$};
  }

  \fill[fill=white] (8,-.1) rectangle (9.6, 2.9);

  \node at (8.8, 1.4) {\Large$\cdots$};

  \draw [decorate,decoration={brace,amplitude=10pt}]
  (0,-0.5) -- (11,-0.5) node [scale=1, black,midway, above, yshift=20pt]
        {$p$ copies of $R_d$, $p+1$ copies of $R_\times$};

  \draw [decorate,decoration={brace,amplitude=10pt}]
  (11,3.3) -- (0,3.3) node [scale=1, black,midway, below, yshift=-20pt]
        {$q$ copies of $R_c$, $q+1$ copies of $R_\times$};

  \node[BoxLabelStyle] at (5.5,5.5) {\small $G^{10}_{n,p,q}$ or $\Gdpacqa$};
\end{scope}

\end{scope}
\end{tikzpicture}
\caption{\label{fig:Gamma-npq-horiz} The long horizontal tiles in $\Gamma_{n,p,q}$.}
\end{figure} 

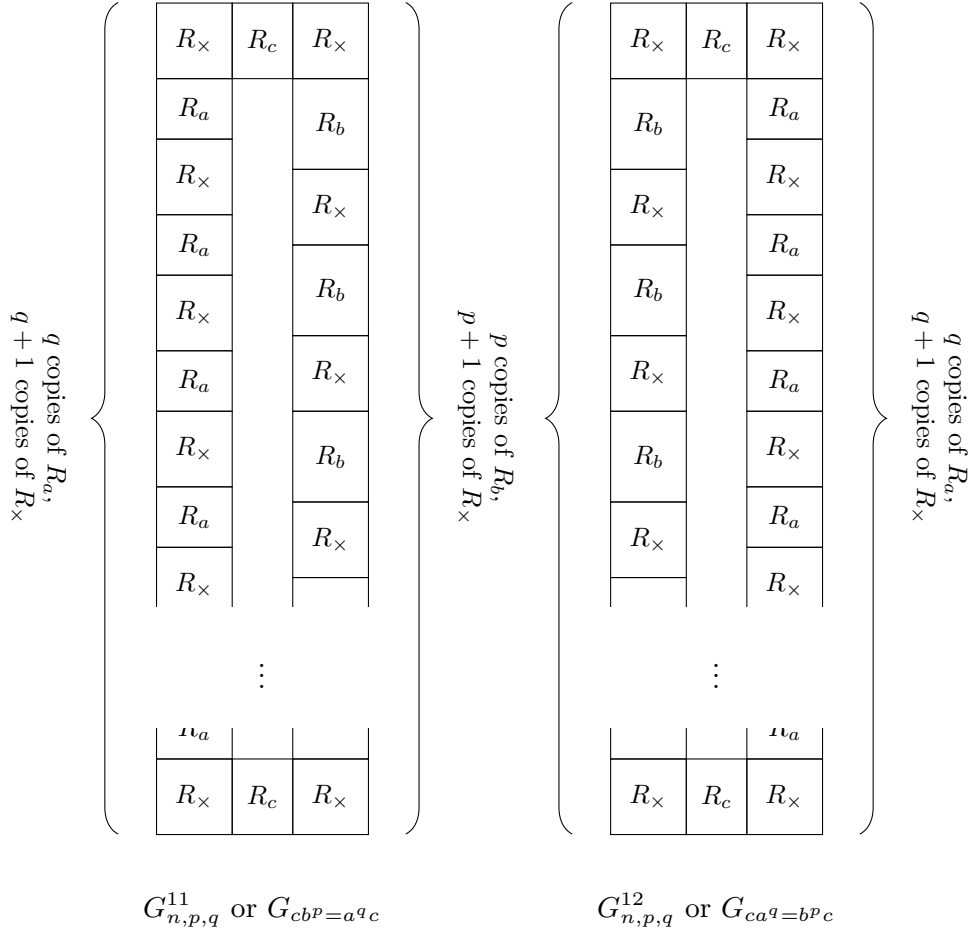
\begin{figure}[H] 
\begin{tikzpicture}
\begin{scope}[scale=1,yscale=-1]
\begin{scope}[shift={(0,0)}]  
  \foreach \loc in {(0.0,0),(0.0,1.8),(0.0,3.6),(0.0,5.4),(0.0,7.2),(0.0,10),
                    (1.8,0),(1.8,2.2),(1.8,4.4),(1.8,6.6),(1.8,10)} {
    \node[XBoxStyle] at \loc {$R_\times$};
  }
  \foreach \loc in {(1,0),(1,10)} {
    \node[CBoxStyle] at \loc {$R_c$};
  }
  \foreach \loc in {(0,1),(0,2.8),(0,4.6),(0,6.4),(0,9.2)} {
    \node[ABoxStyle] at \loc {$R_a$};
  }
  \foreach \loc in {(1.8,1),(1.8,3.2),(1.8,5.4),(1.8,7.6),(1.8,8.8)} {
    \node[BBoxStyle] at \loc {$R_b$};
  }

  \fill[fill=white] (-.1,8) rectangle (2.9,9.6);

  \node at (1.4,8.8) {\Large$\vdots$};

  \draw [decorate,decoration={brace,amplitude=10pt}]
  (-0.5,11) -- (-0.5,0) node [scale=1, black,midway, left, xshift=-30pt, rotate=-90, anchor=center]
        {\begin{tabular}{c}$q$ copies of $R_a$, \\ $q+1$ copies of $R_\times$\end{tabular}};

  \draw [decorate,decoration={brace,amplitude=10pt}]
  (3.3,0) -- (3.3,11) node [scale=1, black,midway, right, xshift=30pt, rotate=-90, anchor=center]
        {\begin{tabular}{c}$p$ copies of $R_b$, \\ $p+1$ copies of $R_\times$\end{tabular}};

  \node[BoxLabelStyle] at (1.4,12) {\small $G^{11}_{n,p,q}$ or $\Gcbpcaq$};
\end{scope}

\begin{scope}[shift={(6,0)}]  
  \foreach \loc in {(1.8,0),(1.8,1.8),(1.8,3.6),(1.8,5.4),(1.8,7.2),(1.8,10),
                    (0,0),(0,2.2),(0,4.4),(0,6.6),(0,10)} {
    \node[XBoxStyle] at \loc {$R_\times$};
  }
  \foreach \loc in {(1,0),(1,10)} {
    \node[CBoxStyle] at \loc {$R_c$};
  }
  \foreach \loc in {(1.8,1),(1.8,2.8),(1.8,4.6),(1.8,6.4),(1.8,9.2)} {
    \node[ABoxStyle] at \loc {$R_a$};
  }
  \foreach \loc in {(0,1),(0,3.2),(0,5.4),(0,7.6),(0,8.8)} {
    \node[BBoxStyle] at \loc {$R_b$};
  }

  \fill[fill=white] (-.1,8) rectangle (2.9,9.6);

  \node at (1.4,8.8) {\Large$\vdots$};

  \draw [decorate,decoration={brace,amplitude=10pt}]
  (-0.5,11) -- (-0.5,0) node [scale=1, black,midway, left, xshift=-30pt, rotate=-90, anchor=center]
        {};

  \draw [decorate,decoration={brace,amplitude=10pt}]
  (3.3,0) -- (3.3,11) node [scale=1, black,midway, right, xshift=30pt, rotate=-90, anchor=center]
        {\begin{tabular}{c}$q$ copies of $R_a$, \\ $q+1$ copies of $R_\times$\end{tabular}};

  \node[BoxLabelStyle] at (1.4,12) {\small $G^{12}_{n,p,q}$ or $\Gcaqcbp$};
\end{scope}

\end{scope}
\end{tikzpicture}

\caption{The long vertical tiles in $\Gamma_{n,p,q}$.}
\label{fig:Gamma-npq-vert}
\end{figure} 

\section{The long-tile energy obstruction}\label{sec:energy}

In this section, we will show that there is no graph homomorphism from the long-tile to 31 tournaments.

\subsection{The energy function}

Let \(D\) be a digraph. An \emph{energy} function on \(D\) is a map
\[
                    \eta:A(D)\longrightarrow\{0,1\}.
\]
A \emph{directed diamond} is a quadruple \((a,b,c,d)\) satisfying
\[
       a\to b,\quad a\to c,\quad b\to d,\quad c\to d.
\]
The energy is \emph{diamond-compatible} if
\begin{equation}\label{eq:diamond}
 \eta(a,b)+\eta(b,d)=\eta(a,c)+\eta(c,d)
\end{equation}
for every directed diamond.

For \(i\in\{0,1\}\), let \(E_i\) be the spanning subdigraph formed by
the arcs of energy \(i\). We call a digraph \emph{coprime-free} if it
does not contain directed closed walks of two relatively prime lengths.
We call \(\eta\) \emph{admissible} when it is diamond-compatible and
both energy subdigraphs \(E_0,E_1\) are coprime-free.

For example, an acyclic digraph is coprime-free.  We will also use subdigraphs whose
only directed simple cycle is a triangle, i.e., every directed closed walk has length divisible
by \(3\), so they too are coprime-free.

\begin{lemma}[Square cancellation]\label{lem:square}
Let \(c\) be an oriented coloring from an oriented rectangular grid to a digraph \(D\)
carrying a diamond-compatible energy \(\eta\). Given an arc on the oriented grid \(x\to y\) we define the pulled-back energy
\(\eta(c(x),c(y))\). We define the total energy of a forward directed path to be the sum of the energies of its arc. Then any two forward directed paths on the grid with
the same endpoints have the same total energy.
\end{lemma}

\begin{proof}
Interchanging one horizontal step and one vertical step replaces one
side of a directed diamond by the other.  Equation
\eqref{eq:diamond} preserves the energy.  Any two forward paths in a
rectangle are related by a sequence of these interchanges.
\end{proof}

Intuitively, the total energy it takes to move from a vertex to another one does not depend on the choice of paths. That is why we call it energy.

\begin{proposition}[Long-tile obstruction]\label{prop:long}
If a digraph \(D\) admits an admissible energy, then
\[
                 \GammaT_{n,p,q}\not\longrightarrow D
\]
for every \(1\le n<p,q\) with \(\gcd(p,q)=1\).
\end{proposition}

\begin{proof}
Suppose \(c:\GammaT_{n,p,q}\to D\) is a graph homomorphism. Restrict to the long tile
\(T_{c^q a=ad^p}\).  Let \(\gamma\) be the directed closed color walk
of length \(p\) occurring on each \(R_\times ,R_c\)-boundary block, and let
\(\delta\) be the corresponding directed closed walk of length \(q\)
on each \(R_\times ,R_d\)-block, and let
\(\alpha\) be the corresponding directed closed walk of length \(p\)
on each \(R_\times ,R_a\)-block.  Write \(E(\gamma)\), \(E(\delta)\) and $E(\alpha)$ for their
total energy respectively.
\begin{figure}[htpb] 
\begin{tikzpicture}
\begin{scope}[scale=1,yscale=-1]
\begin{scope}[shift={(0,0)}]  
  \foreach \loc in {(0,0.0),(1.8,0.0),(3.6,0.0),(5.4,0.0),(7.2,0.0),(10,0.0),
                    (0,1.8),(2.2,1.8),(4.4,1.8),(6.6,1.8),(10,1.8)} {
    \node[XBoxStyle] at \loc {$R_\times$};
  }
  \foreach \loc in {(0,1),(10,1)} {
    \node[ABoxStyle] at \loc {$R_a$};
  }
  \foreach \loc in {(1,0),(2.8,0),(4.6,0),(6.4,0),(9.2,0)} {
    \node[CBoxStyle] at \loc {$R_c$};
  }
  \foreach \loc in {(1,1.8),(3.2,1.8),(5.4,1.8),(7.6,1.8),(8.8,1.8)} {
    \node[DBoxStyle] at \loc {$R_d$};
  }

  \node at (8.8, 1.4) {\Large$\cdots$};

\node at (0.9,-0.2) {$\gamma$};
\node at (-0.2, 0.9) {$\alpha$};
\node at (3.3, 1.5) {$\delta$};

\draw[fill,radius=0.05,black] (0,0) circle;
\draw[fill,radius=0.05,black] (1.8,0) circle;
\draw[fill,radius=0.05,black] (3.6,0) circle;
\draw[fill,radius=0.05,black] (5.4,0) circle;
\draw[fill,radius=0.05,black] (7.2,0) circle;
\draw[fill,radius=0.05,black] (10,0) circle;
\draw[fill,radius=0.05,black] (0,1.80) circle;
\draw[fill,radius=0.05,black] (2.20,1.80) circle;
\draw[fill,radius=0.05,black] (4.40,1.80) circle;
\draw[fill,radius=0.05,black] (6.60,1.80) circle;
\draw[fill,radius=0.05,black] (10,1.80) circle;

 \draw[draw=black, ultra thick] (0,0) to (10,0);
 \draw[draw=black, ultra thick] (0,1.80) to (10,1.80);
 \draw[draw=black, ultra thick] (0,0) to (0,1.80);
 \draw[draw=black, ultra thick] (10,0) to (10,1.80);

 \fill[fill=white] (8,-.1) rectangle (9.6, 2.9);



\end{scope}

\end{scope}
\end{tikzpicture}
\caption{\label{fig:Long} The long horizontal tile $\Tcqadpa$. Consider two forward paths indicated by the thick line.}
\end{figure}
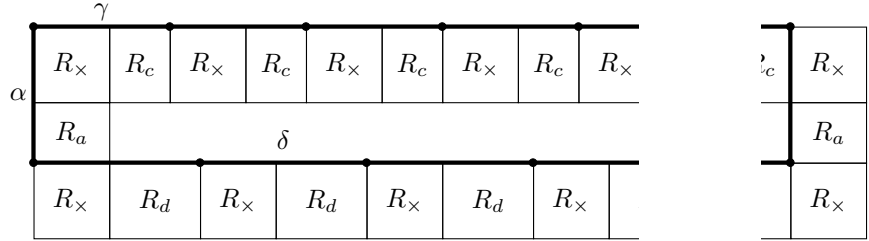

Apply Lemma \ref{lem:square} to the two boundary routes from the
upper-left corner to the lower-right corner, we obtain $qE(\gamma)+E(\alpha)=pE(\delta)+E(\alpha)$. The equal side paths
cancel, giving
\begin{equation}\label{eq:long}
                        qE(\gamma)=pE(\delta).
\end{equation}
Since \(\gcd(p,q)=1\), equation \eqref{eq:long} implies
\(p\mid E(\gamma)\) and \(q\mid E(\delta)\).  But
\[
              0\le E(\gamma)\le p,\qquad
              0\le E(\delta)\le q.
\]
Hence either
\[
        E(\gamma)=E(\delta)=0
\]
or
\[
        E(\gamma)=p,\qquad E(\delta)=q.
\]
In the first case \(E_0\) contains directed closed walks of the
relatively prime lengths \(p,q\); in the second case the same is true
of \(E_1\).  Both alternatives contradict admissibility.
\end{proof}

\begin{remark}
    For the digraph $D$ with $V(D)=\mathbb Z_n,n>2$ and $$(i,j)\in A(D)\Longleftrightarrow j-i\equiv1\text{ or }2(\bmod n)$$ there is an admissible energy function $\eta:A(D)\to\{0,1\}$
    \[
 \eta(i,j)=
 \begin{cases}
 0,&j-i\equiv 2(\bmod n),\\
 1,&j-i\equiv 1(\bmod n).
 \end{cases}
\]
    When $n=3$ the above digraph is the undirected complete graph on 3 vertices $K_3$, which is a directed graph with 6 arcs. So the continuous chromatic number of undirected Schreier graph is greater than 3, this extends \cite[Theorem 1.5.1]{GJKS2025} and \cite[Theorem 3.1.1]{GJKS2025}. 
\end{remark}
From now on, we will focus on graph homomorphism to tournaments. Since the digraph $\GammaT_{n,p,q}$ is strongly connected, the image of it lies in a strong connected component, so we will focus on strong tournaments.
\subsection{27 strong tournaments with the order energy}

The majority of the finite tournament classification is handled by one
simple criterion. Identify the vertex set $V(T)$ with the order
\[
                         0<1<\cdots<m-1.
\]
A tournament is determined by its set \(B\) of backward arcs
\(j\to i\) with \(j>i\), and every pair not in \(B\) is oriented forward.
Define the order energy for every $(u,v)\in A(T)$,
\[
 \eta(u,v)=
 \begin{cases}
 0,&u<v,\\
 1,&u>v.
 \end{cases}
\]
That is, an arc has zero energy if and only if it is forward. Both energy subdigraphs are acyclic since energy-zero arcs move forward in
the order and energy-one arcs move backward. So both energy subdigraphs are coprime-free.

For distinct \(a,d\), put
\[
        I(a,d)=N^+(a)\cap N^-(d),
\]
the set of middle vertices of a directed two-step path
\(a\to x\to d\).

\begin{lemma}[Interval criterion]\label{lem:interval}
The order energy is diamond-compatible if and only if, for every ordered
pair \(a,d\), the set \(I(a,d)\) lies entirely inside the open interval
between \(a\) and \(d\), or entirely outside that interval.
\end{lemma}

\begin{corollary}\label{cor:order}
The order energy on those tournaments satisfying the interval criterion
is admissible.
\end{corollary}

 We show that there are 27 strong tournaments from $T_1$ to $T_{27}$ satisfying the interval criterion, see Appendix \ref{app:orders} for the specific certificate. All remaining pairs are vacuous or have a single intermediate vertex.

\subsection{4 strong tournaments with non-order energy}

We show that there are 4 more tournaments admitting an admissible energy function. We list them as follows.

\begin{table}[ht]
\centering
\small
\begin{tabular}{@{}ccllc@{}}
\toprule
case&backward set \(B\)&energy-one arcs&only cycle in \(E_0\)&code\\
\midrule
\(T_{28}\)&\(\{42,50\}\)&\(\{50\}\)&\(2\to3\to4\to2\)&80\\
\(T_{29}\)&\(\{42,50,51\}\)&\(\{50,51\}\)&\(2\to3\to4\to2\)&88\\
\(T_{30}\)&\(\{31,50\}\)&\(\{50\}\)&\(1\to2\to3\to1\)&176\\
\(T_{31}\)&\(\{31,40,50\}\)&\(\{01,02,03\}\)&\(1\to2\to3\to1\)&377\\
\bottomrule
\end{tabular}
\caption{The four non-order energy certificates.  All unlisted arcs
have energy \(0\), and \(E_1\) is acyclic.}
\label{tab:special}
\end{table}

\begin{proposition}\label{prop:31}
Each of the 4 tournaments represented in Table \ref{tab:special} admits an admissible energy.
\end{proposition}

\begin{proof}

For \(T_{28}\), $N^+(5)=\{0\}$ and $N^-(0)=\{5\}$, so the only energy-one arc \(5\to0\) only occurs in directed diamond $(a,b,c,d)$ with $b=c$.  For \(T_{30}\) the same statement holds.  Hence these two tournaments are diamond-compatible.

For \(T_{29}\), $N^-(0)=N^-(1)=\{5\}$, so the only diamonds meeting an energy-one arc have the form either $(a,b,c,d)\quad b=c$ or
\[
             (5,0,1,d),\qquad d\in\{2,3,4\},
\]
up to interchanging the two middle vertices.  Each of their two
length-two paths contains exactly one energy-one arc.

For \(T_{31}\), $N^-(0)=\{4,5\}$ and $N^+(5)=\{0\},N^+(4)=\{0,5\}$, so the only diamonds meeting an energy-one arc have the form either $(a,b,c,d)\quad b=c$ or vertex \(a=0\), middle vertices \( b,c\in \{1,2,3\}\), and vertex \(d\in\{4,5\}\). Again each path
contains exactly one energy-one arc.

Thus all four energies are diamond-compatible.  In each row \(E_1\) is
visibly acyclic.

For \(T_{28}\), $N^+(5)=\{0\}$ and $N^-(0)=\{5\}$, and the arc $5\to 0$ has energy-one, so $0$ and $5$ are not in any strongly connected component of $E_0$. And $N^-(1)=\{0\}$, so $1$ is not in any strongly connected component of $E_0$.

For \(T_{29}\), $N^+(5)=\{0,1\}$ and $N^-(0)=\{5\}$, and the arcs $5\to 0$ and $5\to 1$ have energy-one, so $0$ and $5$ are not in any strongly connected component of $E_0$. And $N^-(1)=\{0,5\}$, so $1$ is not in any strongly connected component of $E_0$.

For \(T_{30}\), $N^+(5)=\{0\}$ and $N^-(0)=\{5\}$, and the arc $5\to 0$ has energy-one, so $0$ and $5$ are not in any strongly connected component of $E_0$. And $N^+(4)=\{5\}$, so $4$ is not in any strongly connected component of $E_0$.

For \(T_{31}\), $N^+(0)=\{1,2,3\}$, and the arcs $0\to 1,0\to 2,0\to 3$ have energy-one, so $0$ is not in any strongly connected component of $E_0$. And $N^+(5)=\{0\}$, so $5$ is not in any strongly connected component of $E_0$. And $N^+(4)=\{0,5\}$, so $4$ is not in any strongly connected component of $E_0$.

The triangle shown in the table is its only directed simple cycle.  Hence every closed walk in \(E_0\)
has length divisible by \(3\).  Both energy subdigraphs are therefore
coprime-free for each tournament. So each energy function in the list is admissible.
\end{proof}

\begin{corollary}\label{cor:31blocked}
None of these thirty-one strong six-vertex tournaments is a homomorphic
image of any \(\GammaT_{n,p,q}\) with coprime \(p,q\).
\end{corollary}

\section{The torus-tile marker obstruction}\label{sec:markers}

In this section, we will show that there is no graph homomorphism from the torus digraph $\Cp_p\square\Cp_q$ to the remaining tournaments. The four remaining strong six-vertex tournaments form two pairs under
reversal.  Representatives \(T_A=T_{32},T_B=T_{33}\) are specified by their
out-neighborhoods:
\begin{equation}\label{eq:TA}
\begin{aligned}
N^+_{T_A}(0)&=\{2,5\},&
N^+_{T_A}(1)&=\{0\},&
N^+_{T_A}(2)&=\{1\},\\
N^+_{T_A}(3)&=\{0,1,2\},&
N^+_{T_A}(4)&=\{0,1,2,3\},&
N^+_{T_A}(5)&=\{1,2,3,4\},
\end{aligned}
\end{equation}
and
\begin{equation}\label{eq:TB}
\begin{aligned}
N^+_{T_B}(0)&=\{1,4,5\},&
N^+_{T_B}(1)&=\{3\},&
N^+_{T_B}(2)&=\{0,1\},\\
N^+_{T_B}(3)&=\{0,2\},&
N^+_{T_B}(4)&=\{1,2,3\},&
N^+_{T_B}(5)&=\{1,2,3,4\}.
\end{aligned}
\end{equation}
The other two classes are \(T_A^{\mathrm{op}}=T_{34}\) the reversal digraph of $T_A$ and
\(T_B^{\mathrm{op}}=T_{35}\) the reversal digraph of $T_B$.

For a tournament \(T\), let
\[
   \PathTwo(T)=\{(x,y,z):x\to y\to z\}
\]
be the set of directed paths of length two.

For \(T_A\), define the marked triples
\begin{equation}\label{eq:MA}
\begin{aligned}
\mathcal M_{T_A}=\{&
(0,2,1),(0,5,1),(0,5,2),(0,5,3),(0,5,4),\\
&(3,2,1),(4,2,1),(4,3,1)\}\subseteq\PathTwo(T_A).
\end{aligned}
\end{equation}
For \(T_B\), it is shorter to list the unmarked triples:
\begin{equation}\label{eq:UB}
\begin{aligned}
\mathcal U_{T_B}={}&
 \{(x,0,z):x\in\{2,3\},\ z\in\{1,4,5\}\}\\
&{}\cup\{(x,2,1):x\in\{3,4,5\}\}
 \cup\{(5,4,1)\}\subseteq\PathTwo(T_B),
\end{aligned}
\end{equation}
and put
\[
                    \mathcal M_{T_B}=\PathTwo(T_B)\setminus\mathcal U_{T_B}.
\]

If
\[
 r=(r_i)_{i\in\mathbb Z/p\mathbb Z}
\]
is a directed closed walk of length $p$ in \(T=T_A\) or \(T_B\), define its marker set
\[
 M(r)=\{i:(r_i,r_{i+1},r_{i+2})\in\mathcal M_T\}.
\]

\subsection{The local transfer law}

\begin{lemma}[Marker shift]\label{lem:shift}
Let \(T\in\{T_A,T_B\}\).  Suppose \(r=(r_i)\) and \(s=(s_i)\) are
directed rows satisfying
\[
       r_i\to r_{i+1},\qquad s_i\to s_{i+1},
       \qquad r_i\to s_i
\]
for every \(i\).  Then
\[
                         M(s)=M(r)-1.
\]
Equivalently,
\[
       (s_i,s_{i+1},s_{i+2})\in\mathcal M_T
       \quad\Longleftrightarrow\quad
       (r_{i+1},r_{i+2},r_{i+3})\in\mathcal M_T.
\]
\end{lemma}

\begin{proof}
This is a local statement about a directed \(2\times4\) rectangle:
\[
\begin{array}{ccccccc}
r_i&\to&r_{i+1}&\to&r_{i+2}&\to&r_{i+3}\\
\downarrow&&\downarrow&&\downarrow&&\downarrow\\
s_i&\to&s_{i+1}&\to&s_{i+2}&\to&s_{i+3}.
\end{array}
\]

For \(T_A\), one can use the out-neighborhood list (\ref{eq:TA}) to show that the number of $\mathcal P_2(T_A)$ is 28. $\mathcal P_2(T_A)$ splits into the following three
transfer classes:
\[
\mathcal M_A,
\]
\[
\begin{aligned}
\mathcal A_0={}&
 \{(x,0,z):x\in\{1,3,4\},\ z\in\{2,5\}\}\\
 &{}\cup\{(4,3,2)\},
\end{aligned}
\]
and
\[
\begin{aligned}
\mathcal A_-={}&
 \{(2,1,0),(3,1,0),(4,1,0),(4,3,0),(5,1,0),(5,2,1)\}\\
&{}\cup\{(5,3,z):z\in\{0,1,2\}\}\\
&{}\cup\{(5,4,z):z\in\{0,1,2,3\}\}.
\end{aligned}
\]
These sets have sizes \(8,7,13\), respectively, and exhaust the
28 directed triples of \(T_A\).

For \(T_B\), there are only two transfer classes:
\[
                         \mathcal M_B
                         \quad\text{and}\quad
                         \mathcal U_B.
\]
Using the out-neighborhood lists \eqref{eq:TA} and \eqref{eq:TB}, one
checks a possible rectangle column by column: the lower triple and the
shifted upper triple always remain in the same displayed class. The verification is tedious like playing a Sudoku. This
completes the proof of local transfer law.
\end{proof}

\subsection{Every row has both types}

\begin{lemma}[Row nontriviality]\label{lem:row}
For \(T\in\{T_A,T_B\}\), every directed closed walk \(r\) of length $p$ satisfies
\[
                   \varnothing\ne M(r)\ne\mathbb Z/p\mathbb Z.
\]
\end{lemma}

\begin{proof}
For \(\epsilon\in\{0,1\}\), form a de Bruijn digraph \(D_\epsilon(T)\).
Its vertices \(V(D_\epsilon(T))=A(T)\), and it has an arc
\[
                  (x,y)\longrightarrow(y,z)
\]
when \((x,y,z)\) is a directed path whose marker value is
\(\epsilon\). $\epsilon=1$ if $(x,y,z)\in\mathcal{M}_T$ and $\epsilon=0$ if $(x,y,z)\notin\mathcal{M}_T$. We claim that both $D_\epsilon(T)$ are acyclic. Then, the directed closed walk $r$ is not fully marked $M(r)\neq\mathbb Z/ p\mathbb Z$, since $D_1(T)$ is acyclic; the directed closed walk $r$ is not fully unmarked $M(r)\neq\varnothing$ since $D_0(T)$ is acyclic.

For each marker value $\epsilon$, the following level sets are a partition of the vertex set for the corresponding
de Bruijn digraph $D_\epsilon(T)$. One can verify that every de Bruijn arc goes from a lower-numbered level
to a higher-numbered level, and the highest-numbered level is contained in the set of sinks.

For marked triples of \(T_A\):
\[
\begin{aligned}
L_0={}&\{02,05,10,30,32,40,41,42,43\},\\
L_1={}&\{21,31,51,52,53,54\}.
\end{aligned}
\]
For unmarked triples of \(T_A\):
\[
\begin{aligned}
L_0={}&\{51,52,53,54\},\\
L_1={}&\{21,31,40,41,42,43\},\\
L_2={}&\{10,30,32\},\\
L_3={}&\{02,05\}.
\end{aligned}
\]
For marked triples of \(T_B\):
\[
\begin{aligned}
L_0={}&\{01,04,05,21\},\\
L_1={}&\{41,51,52,53,54\},\\
L_2={}&\{13,42,43\},\\
L_3={}&\{30,32\},\\
L_4={}&\{20\}.
\end{aligned}
\]
For unmarked triples of \(T_B\):
\[
\begin{aligned}
L_0={}&\{13,20,30,32,42,43,51,52,53,54\},\\
L_1={}&\{01,04,05,21,41\}.
\end{aligned}
\]
This completes the proof of claim.
\end{proof}

\begin{proposition}[Torus obstruction]\label{prop:torus}
If \(p,q\ge3\) and \(\gcd(p,q)=1\), then
\[
     \Cp_p\square\Cp_q\not\longrightarrow T_A,
     \qquad
     \Cp_p\square\Cp_q\not\longrightarrow T_B.
\]
The same holds for the reversals of these tournaments.
\end{proposition}

\begin{proof}
Suppose
\[
               f:\Cp_p\square\Cp_q\longrightarrow T
\]
is a graph homomorphism to \(T=T_A\) or \(T_B\).  Let \(r^{(j)}\) be the directed closed walk which is the image of the \(j\)-th horizontal
row and set \(M_j=M(r^{(j)})\).  Lemma \ref{lem:shift} gives
\[
                         M_{j+1}=M_j-1.
\]
After traversing all \(q\) rows,
\[
                         M_0=M_q=M_0-q
\]
as subsets of \(\mathbb Z/p\mathbb Z\).  Since \(\gcd(p,q)=1\),
translation by \(q\) acts transitively on \(\mathbb Z/p\mathbb Z\).
Its only invariant subsets are the empty set and the whole set.  This
contradicts Lemma \ref{lem:row}.

For the reversal, if \(f:\Cp_p\square\Cp_q\to T^{\mathrm{op}}\), then
\[
                         g(i,j)=f(-i,-j)
\]
is a homomorphism to \(T\), giving the same contradiction.
\end{proof}

\begin{corollary}\label{cor:hardblocked}
For every \(1\le n<p,q\) with \(\gcd(p,q)=1\),
\[
 \GammaT_{n,p,q}\not\longrightarrow
 T_A,T_A^{\mathrm{op}},T_B,T_B^{\mathrm{op}}.
\]
\end{corollary}


\begin{theorem}\label{thm:lower}
There is no continuous graph homomorphism from \(\Shift\) to any strong tournament on 6 vertices. Thus
\[
                           \chic(\Shift)>6.
\]
\end{theorem}

\section{The seven-color upper bound}\label{sec:upper}

Let \(R_7\) be the tournament on \(\{0,1,\ldots,6\}\) with arcs
\[
\begin{aligned}
&0\to1,\quad 0\to3,\quad 0\to6,\\
&1\to2,\quad 1\to5,\\
&2\to0,\quad 2\to4,\quad 2\to6,\\
&3\to1,\quad 3\to2,\quad 3\to4,\quad 3\to6,\\
&4\to0,\quad 4\to1,\quad 4\to5,\\
&5\to0,\quad 5\to2,\quad 5\to3,\\
&6\to1,\quad 6\to4,\quad 6\to5.
\end{aligned}
\]

\begin{proposition}\label{prop:upper}
There is a homomorphism
\[
                         \GammaT_{1,3,4}\longrightarrow R_7.
\]
Consequently,
\[
                         \chic(\Shift)\le7.
\]
\end{proposition}

\begin{proof}
As illustrated in Appendix \ref{app:coloring}, there is an oriented coloring of $\GammaT_{1,3,4}$ with 7 colors, rows are written top-to-bottom and columns left-to-right.  Horizontal
arcs point right and vertical arcs point down. Theorem \ref{thm:twelve} lifts the
finite homomorphism to a continuous homomorphism \(\Shift\to R_7\).
\end{proof}

\begin{proof}[Proof of Theorem \ref{thm:main}]
Combine Theorems \ref{thm:lower} and Proposition \ref{prop:upper}.
\end{proof}

\section{The Borel oriented chromatic number of abelian group actions}
In this section, we show that the Borel oriented chromatic number of the directed Schreier graph
$\vec F(2^{\mathbb Z^n})$, $n>1$ is
\[
             \chi_{Bo}(\vec F(2^{\mathbb Z^n}))=5.
\]
\subsection{The lower bound}
\begin{proposition}\label{prop:Borellower}
    $\chi_{Bo}(\vec F(2^{\mathbb Z^n}))>4,\ n>1.$
\end{proposition}
\begin{proof}
    An easy ergodicity argument shows that $\chi_{Bo}(\vec F(2^{\mathbb Z^n}))>3,\ n>1.$

    There is only one strong tournament on four vertices. We label it by
\(\{0,1,2,3\}\), with arcs
\[
0\to1,\quad 0\to2,\quad 1\to2,\quad 1\to3,\quad 2\to3,\quad 3\to0.
\]
Now we assume that there is a Borel graph homomorphism $c: F(2^{\mathbb Z^n})\to\{0,1,2,3\}$. Note that \(3\) has unique out-neighbor \(0\), and \(0\) has unique in-neighbor \(3\), we have that
\[
c(e_1\cdot x)=0
\quad\Longleftrightarrow\quad
c(e_2\cdot x)=0.
\]
So $c^{-1}(0)$ is invariant under ergodic group action $\mathbb Z\curvearrowright F(2^{\mathbb Z^n})$ by $1\cdot x=(e_1-e_2)x$, then $c^{-1}(0)$ is either meager or comeager, both are not possible.
\end{proof}
\subsection{The upper bound}

We show that there is a Borel graph homomorphism from the directed Schreier graph $\vec F(2^{\mathbb Z^n})$ to the regular tournament on 5 vertices $P_5$ by using the Borel toast structure where $V(P_5)=\mathbb Z_5$ and 
$$(i,j)\in A(P_5)\Longleftrightarrow j-i\equiv 1,2(\bmod 5).$$
\begin{theorem}[\cite{GJKS}]\label{thm:toast}
Fix \(R\in\mathbb N\), say $R=100$.  There is a Borel family \(\cT\) of finite connected
subsets such that:
\begin{enumerate}[label=(\roman*)]
\item any two members of \(\cT\) are either disjoint or one contains the other;
\item every finite connected subset of an orbit is contained in some member of
\(\cT\);
\item for every \(D\in\cT\), if \(E_1,\dots,E_m\) are the maximal proper
members of \(\cT\) contained in \(D\), then the boundaries of $D,E_1,\cdots,E_m$
are pairwise separated at least \(R\)-distance.
\end{enumerate}
\end{theorem}

Now we fix some notations.

For \(z=(z_1,\dots,z_n)\in\mathbb Z^n\), write
$
\sigma(z)=\Sigma z_i.
$
A function \(H:Q\to\mathbb Z\), where \(Q\subseteq\mathbb Z^n\), gives an oriented
\(P_5\)-coloring by modulo \(5\), if
\[
H(z+e_i)-H(z)\in\{1,2\}
\]
whenever \(z,z+e_i\in Q\).  Let
$
F(z)=H(z)-\sigma(z).
$
Then the condition becomes
\[
F(z+e_i)-F(z)\in\{0,1\}.
\]

The \emph{model height} is
\[
F_0(z)=\left\lfloor\frac{\sigma(z)}2\right\rfloor,
\qquad
H_0(z)=\sigma(z)+F_0(z).
\]
It is easy to see that \(H_0(\bmod 5)\) is an oriented \(P_5\)-coloring on the directed Cayley graph.

Let \(Q\subseteq\mathbb Z^n\) be a subset of vertices.  For \(u,v\in Q\), we define a quasi-metric
\(\delta_Q(u,v)\) as follows. $\delta(u,u+e_i)=1$ and $\delta(u,u-e_i)=0$.
$$\delta_Q(u,v)=\min\{\Sigma\delta(u_i,u_{i+1}):u_0=u,u_m=v,u_0,u_1,\cdots,u_m\text{ is an undirected path}\}.$$

\(\delta_Q(u,v)=\infty\) if no such path exists. For example, in the full lattice,
$
\delta_{\mathbb Z^n}(u,v)
 =
\sum_{i=1}^n \max(v_i-u_i,0).
$

The following lemmas are the key observations in the proof.

\begin{lemma}[extension criterion]\label{lem:criterion}
Let \(S\subseteq Q\subseteq \mathbb Z^n\) be finite subsets and assume that every connected component of
$Q$ meets $S$. and \(f:S\to\mathbb Z\). Then \(f\) can be extended to a function \(F:Q\to\mathbb Z\)
satisfying
\[
F(z+e_i)-F(z)\in\{0,1\}
\]
for all \(z,z+e_i\in Q\), if and only if
\[
f(v)-f(u)\leq \delta_Q(u,v)
\]
for each ordered pair \((u,v)\in S^2\).
\end{lemma}

\begin{proof}
The necessity is easy.

Conversely, define
\[
F(z)=\min_{s\in S}\bigl(f(s)+\delta_Q(s,z)\bigr).
\]
The minimum is finite because every component of $Q$ meets $S$. 

For $t\in S$,
take $s=t$ then we have $F(t)\leq f(t)$. On the other hand,
\[
f(t)-f(s)\leq\delta_Q(s,t)
\]
for every $s\in S$, taking minima over $s\in S$ gives $f(t)\leq F(t)$. Hence $F(t)=f(t)$.

If \(z,z+e_i\in Q\), then
\[
\delta_Q(s,z+e_i)\leq \delta_Q(s,z)+1
\quad
\text{ and }
\quad
\delta_Q(s,z)\leq \delta_Q(s,z+e_i).
\]
Taking minima over \(s\in S\)
gives
\[
F(z)\leq F(z+e_i)\leq F(z)+1.
\]

\end{proof}

For a based point \(a\in\mathbb Z^n\), let
\[
H_a(z)
 =
\sigma(z-a)+\left\lfloor\frac{\sigma(z-a)}2\right\rfloor,
\]
and
\[
f_a(z)
 =
-\sigma(a)+\left\lfloor\frac{\sigma(z)-\sigma(a)}{2}\right\rfloor.
\]
Then,
\[
H_a(z)=\sigma(z)+f_a(z).
\]

\begin{lemma}[phase lifting]\label{lem:phase-lifting}
For every \(a\in\mathbb Z^n\), there is \(k\in\mathbb Z\) such that
\[
\left|f_a(z)+5k-F_0(z)\right|\leq 3
\]
for every \(z\in\mathbb Z^n\).
\end{lemma}

\begin{proof}
The value
\[
f_a(z)-F_0(z)
=
-\sigma(a)+\left\lfloor\frac{\sigma(z)-\sigma(a)}{2}\right\rfloor
-\left\lfloor\frac{\sigma(z)}2\right\rfloor
\]
depends only on \(\sigma(z)\pmod 2\). Hence, it takes either one value or two
adjacent integer values.  By adding a suitable multiple of \(5\), these one or
two adjacent values can be moved into the interval \([-3,3]\).
\end{proof}

\begin{lemma}[Separated profile extension]\label{lem:separated-extension}
Let $Q\subseteq\Z^n$ be finite, and let
\[
S=S_1\sqcup S_2\sqcup\cdots\sqcup S_m\subseteq Q.
\]
Assume that the $\ell^1$-distance between distinct $S_j$ is at least
$100$. Suppose that $f:S\to\Z$ has the following properties:
\begin{enumerate}[label=\textup{(\roman*)}]
\item on each $S_j$,
      \[
      f(z)=C_j+\left\lfloor\frac{\sigma(z)-A_j}{2}\right\rfloor
      \]
      for some integers $A_j,C_j$;
\item $|f(z)-F_0(z)|\leq3$ for every $z\in S$.
\end{enumerate}
Then $f$ extends to a function $F:Q\to\Z$ satisfying
\[
F(z+e_i)-F(z)\in\{0,1\}
\]
for every positive coordinate edge contained in $Q$.
\end{lemma}

\begin{proof}
By Lemma~\ref{lem:criterion}, it suffices to check
\[
f(v)-f(u)\leq \delta_Q(u,v)
\]
for each ordered pair \((u,v)\in S^2\) in the same connected component of \(Q\).

First suppose \(u,v\in S_j\).  On \(S_j\), the function \(f\) is of form
\[
z\mapsto C_j+\left\lfloor\frac{\sigma(z)-A_j}{2}\right\rfloor,
\]
so,
\[
f(v)-f(u)\leq \delta_{\mathbb Z^n}(u,v)\leq \delta_Q(u,v).
\]

Now suppose \(u\in S_j\), \(v\in S_\ell\), with \(j\neq \ell\).  Put
\[
P=\sum_{r=1}^n \max(v_r-u_r,0),
\qquad
N=\sum_{r=1}^n \max(u_r-v_r,0).
\]
Then
\[
\|u-v\|_1=P+N\geq 100,
\qquad
\sigma(v)-\sigma(u)=P-N.
\]
Since $|f(z)-F_0(z)|\leq3$ for every $z\in S$, we have
\[
f(v)-f(u)
\leq F_0(v)-F_0(u)+6.
\]
And
\[
F_0(v)-F_0(u)
=
\left\lfloor\frac{\sigma(v)}2\right\rfloor
-
\left\lfloor\frac{\sigma(u)}2\right\rfloor
\leq \frac{P-N}{2}+1.
\]
Hence
\[
f(v)-f(u)\leq \frac{P-N}{2}+7.
\]
Since \(P+N\geq 100\), we have
\[
\frac{P-N}{2}+7\leq P.
\]
Finally,
\[
P=\delta_{\mathbb Z^n}(u,v)\leq \delta_Q(u,v).
\]
Thus the extension criterion applies.

Components of \(Q\) not meeting \(S\) can be assigned independently, for instance
by \(F_0\).
\end{proof}

Now we are prepared for the Borel oriented coloring.

\begin{theorem}\label{prop:Borelupper}
    There is a Borel graph homomorphism from the directed Schreier graph $\vec F(2^{\mathbb Z^n})$ to the regular tournament on 5 vertices $P_5$.
\end{theorem}
\begin{proof}
Fix separation $R=100$, and let $\mathcal T$ be a Borel
toast as in Theorem~\ref{thm:toast}. Fix a Borel linear order. For
each piece $D\in\mathcal T$, let $a_D$ be the least point of $D$. Since the action is
free, every $x\in D$ has a unique coordinate
\[
g_D(x)\in\Z^n
\quad\text{such that}\quad
x=g_D(x)\cdot a_D.
\]
Denote
\[
\widetilde D=\{g_D(x):x\in D\}\subseteq\Z^n.
\]
The model coloring associated with $D$ is
\[
p_D(x)
 =\sigma(g_D(x))
  +\left\lfloor\frac{\sigma(g_D(x))}{2}\right\rfloor
  \pmod5.
\]
This is the model coloring $H_0\bmod5$ in coordinates based at $a_D$.

We recursively define, for every $D\in\mathcal T$, a coloring
$
c_D:D\to\Z_5
$
with the following properties:
\begin{enumerate}[label=\textup{(\alph*)}]
\item $c_D$ is a graph homomorphism on the induced subdigraph on $D$;
\item $c_D=p_D$ on $\partial D$, where $\partial D=\{x\in D:\text{dist}(x,F(2^{\mathbb Z^n})\backslash D)=1\}$ is the boundary of $D$;
\item if $E\subsetneq D$ is a maximal proper toast subpiece, then
      $c_D|_E=c_E$.
\end{enumerate}

Let $E_1,\dots,E_m$ be its maximal proper subpieces. By induction, each $c_{E_j}$ has already been defined. Work in the
coordinate system based at $a_D$. Denote
\[
\widetilde E_j=\{g_D(x):x\in E_j\}.
\]
We define the shell
\[
S_D
 =\widetilde D\setminus\bigcup_{j=1}^m
   \{g_D(x):x\in E_j^{\circ}\}, \text{ where }E_j^\circ=E_j\backslash \partial E_j\text{ is the interior of } E_j.
\]
And we define the boundaries
\[
B_D
 =\{g_D(x):x\in\partial_rD\}
  \cup\bigcup_{j=1}^m\{g_D(x):x\in\partial_rE_j\}.
\]
Now we define a function $F$ on $B_D$.

First, on the boundary $\partial_rD$, we define
\[
F(z)=F_0(z)=\left\lfloor\frac{\sigma(z)}2\right\rfloor.
\]
Second, for a subpiece $E_j$, let $h_j\in\Z^n$ be uniquely defined by
\[
a_{E_j}=h_j\cdot a_D.
\]
If $x=z\cdot a_D\in E_j$, then its coordinate relative to $a_{E_j}$ is
$z-h_j$. Let
\[
H_{h_j}(z)
 =\sigma(z-h_j)+\left\lfloor\frac{\sigma(z-h_j)}2\right\rfloor\quad\text{and}\quad
 f_{h_j}(z)=-\sigma(h_j)+\left\lfloor\frac{\sigma(z-h_j)}2\right\rfloor.
\]
By Lemma~\ref{lem:phase-lifting}, choose the least integer $k_j$ such that
$
|f_{h_j}(z)+5k_j-F_0(z)|\leq3
$
for every $z\in\Z^n$. We remark that the choice of $k_j$ is for Borelness.

Then, on the coordinate copy of $\partial_rE_j$, we define
\[
F(z)=f_{h_j}(z)+5k_j.
\]
Adding $5k_j$ does not change the boundary color modulo $5$.

Since the pieces forming $B_D$ are pairwise at
$\ell^1$-distance at least $100$, by lemma~\ref{lem:separated-extension}, $F$ extends to a function on $S_D$ such that
\[
F(z+e_i)-F(z)\in\{0,1\}
\]
for every $z,z+e_i\in S_D$. 

Now we define the coloring on the shell $S_D$
\[
c_D(z\cdot a_D)=\sigma(z)+F(z)\pmod5.
\]
It is a graph homomorphism from the induced subdigraph on $S_D$ to the regular tournament $\Pfive$.

On each interior $E_j^{\circ}$, we define
\[
c_D=c_{E_j}.
\]
On the boundary $\partial E_j$, the coloring definition agrees with $p_{E_j}$,
and property (b) for $c_{E_j}$ gives $c_{E_j}=p_{E_j}$ there. Hence the two
definitions agree on $\partial{E_j}$, and therefore $c_D|_{E_j}=c_{E_j}$. Then all conditions (a)--(c) are satisfied.

We now define the global coloring $c$. For a vertex $x$, let $D_x$ be the
minimal member of $\mathcal T$ containing $x$. Put
\[
c(x)=c_{D_x}(x).
\]
For $D\supseteq D_x$, repeated use of property (c) along the finite chain of
toast pieces between $D_x$ and $D$ gives
\[
c_D|_{D_x}=c_{D_x}.
\]
Thus $c(x)=c_D(x)$ for every toast piece $D$ containing $x$.

Let $x\to e_i\cdot x$ be an arc, there is a piece $D$ containing both $x$ and $e_i\cdot x$, so 
\[
c(e_i\cdot x)-c(x)=c_D(e_i\cdot x)-c_D(x)\in\{1,2\}\pmod5.
\]
Therefore $c$ is a Borel graph homomorphism.
\end{proof}

\begin{proof}[Proof of Theorem~\ref{thm:Borel}]
Combine Theorem~\ref{prop:Borelupper} and Proposition~\ref{prop:Borellower}.
\end{proof}

\section{The graph homomorphism problem}

In \cite[Theorem 4.4.1]{GJKS2025}, Gao, Jackson, Krohne and Seward proved that the set of finite graphs $G$ such that there is a continuous
graph homomorphism from $\Shift$ to $G$ is a $\Sigma^0_1$-complete set. In this section, we show the following.\begin{proposition} The set of finite directed graphs $D$ and the set of finite oriented graphs $O$ that there is a homomorphism from directed Schreier graph to $D$ resp. $O$ are both $\Sigma^0_1$-complete.\end{proposition}

\begin{proof}
    Put$$\mathcal{G}=\{G:G\text{ is a finite graph and there is a continuous graph homomorphism }\Shift\to G\},$$$$\mathcal{D}=\{D:D\text{ is a finite digraph and there is a continuous graph homomorphism }\Shift\to D\},$$$$\mathcal{O}=\{O:O\text{ is a finite oriented graph and there is a continuous graph homomorphism }\Shift\to O\}.$$

    By Theorem \ref{thm:twelve}, we have that both $\mathcal{D}$ and $\mathcal{O}$ are $\Sigma^0_1$. In \cite[Theorem 4.4.1]{GJKS2025}, it was known that $\mathcal{G}$ is a $\Sigma^0_1$-complete set. We show that there is a computable reduction from $\mathcal G$ to $\mathcal{D}$ and $\mathcal{O}$.

    For a graph $G$, let $D_G$ be the symmetric digraph of $G$, that is, an edge is two arcs in both direction. It is easy to see that the map $G\to D_G$ is a reduction from $\mathcal{G}$ to $\mathcal{D}$.

    For a graph $G$, we fix an oriented graph $O$ such that there is a continuous graph homomorphism $\Shift\to O$, let $O_G=G\otimes O$, $V(O_G)=V(G)\times V(O)$ and $$(u,a)\to (v,b)$$iff$$(u,v)\in E(G)\text{ and }(a,b)\in A(O).$$
    $O$ is an oriented graph then $O_G$ is an oriented graph, indeed, assume that both $(u,a)\to(v,b)$ and $(v,b)\to(u,a)$ then $a\to b$ and $b\to a$, a contradiction. It is routine to verify that the map $G\to O_G$ is a reduction from $\mathcal{G}$ to $\mathcal{O}$.
\end{proof}
\section{Future work}\subsection{The graph homomorphism problem}

Put$$\mathcal{T}=\{T:T\text{ is a finite tournament and there is a continuous graph homomorphism }\Shift\to T\}.$$ We do not know if it is $\Sigma^0_1$-hard or not.

One can extend energy functions to integer-valued or abelian group-valued. For example, we can define an energy function valued $\mathbb Z^3$ in the Paley tournament of 7 vertices $\text{PT}(7)$. $V(\text{PT}(7))=\mathbb Z_7$, $(i,j)\in A(\text{PT}(7))$ iff
$$j-i\equiv 1,2,4(\bmod 7).$$ We define $\eta:A(\text{PT}(7))\to\mathbb Z^3$ to be
$$\eta(i,j)=
 \begin{cases}
 (1,0,0),&j-i\equiv 1(\bmod 7),\\
 (0,1,0),&j-i\equiv 2(\bmod 7),\\
 (0,0,1),&j-i\equiv 4(\bmod 7).
 \end{cases}$$
 One can verify that the diamonds start with $0$, $(a,b,c,d)\in\{(0,1,2,3),(0,1,4,5),(0,2,4,6)\}$ satisfy $\eta(ab)+\eta(bd)=\eta(ac)+\eta(cd)$, since the graph is vertex-transitive, the energy function is diamond-compatible. And three energy subdigraphs are $\vec C_7$, so they are coprime-free. A similar proof of \ref{prop:long} shows that there is no continuous graph homomorphism from the directed Schreier graph $\Shift$ to the Paley tournament of 7 vertices $\text{PT}(7)$.
\subsection{Borel and continuous oriented chromatic number}

We do not know the Borel or continuous oriented chromatic numbers of other group shift actions, for example, $7\leq \chi_{Bo}(\vec F(2^{\mathbb F_2}))<\infty$ by \cite{Marks} and $7\leq\chi_{co}(\vec F(2^{\mathbb Z^n}))<\infty,n>2$.

One can show that a Borel oriented graph of bounded degree has finite Borel oriented chromatic number. Indeed, let $G$ be the underlying graph, then $G^2$ has bounded degree where $V(G^2)=V(G)$ and $(x,y)\in E(G^2)\Longleftrightarrow 1\leq d_G(x,y)\leq 2$. By \cite{KST1999}, we have a finite proper coloring of $G^2$. For each vertex $x$, we define
$$S^+(x)=\{c(y):x\to y\},\quad S^-(x)=\{c(y):y\to x\}.$$
Since all neighbors of $x$ have different colors, we have $S^+(x)\cap S^-(x)=\emptyset$. Now we define an oriented graph $H$ with $V(H)=\{(c(x),S^+(x),S^-(x))\}$ and
$$(c(x),S^+(x),S^-(x))\to (c(y),S^+(y),S^-(y))\Longleftrightarrow c(y)\in S^+(x)\text{ and }c(x)\in S^-(y).$$
It is routine to verify that $H$ is a finite oriented graph and $t(x)=(c(x),S^+(x),S^-(x))$ is a Borel oriented coloring.
\begin{acknowledgments}
    The author would like to thank Edward Krohne for consulting on computer programming. The author would like to thank Jiangdong Ai for consulting on graph theory.
\end{acknowledgments}
\appendix

\section{Certificates of interval criterion for 27 tournaments}\label{app:orders}

\begin{multicols}{2}
\raggedcolumns
\setlength{\columnsep}{14pt}
\begin{appcerttable}
\captionof{table}{$T_{1}$: ordered pairs with $\lvert I(a,d)\rvert\ge 2$}
\label{tab:I-T1}
\begin{tabular}{c|c|c}
$(a,d)$ & $I(a,d)$ & position \\ \hline
$(0,3)$ & $\{1,2\}$ & inside \\
$(0,4)$ & $\{1,2,3\}$ & inside \\
$(0,5)$ & $\{1,2,3,4\}$ & inside \\
$(1,4)$ & $\{2,3\}$ & inside \\
$(1,5)$ & $\{2,3,4\}$ & inside \\
$(2,5)$ & $\{3,4\}$ & inside \\
\end{tabular}
\end{appcerttable}

\begin{appcerttable}
\captionof{table}{$T_{2}$: ordered pairs with $\lvert I(a,d)\rvert\ge 2$}
\label{tab:I-T2}
\begin{tabular}{c|c|c}
$(a,d)$ & $I(a,d)$ & position \\ \hline
$(0,3)$ & $\{1,2\}$ & inside \\
$(0,4)$ & $\{1,2,3\}$ & inside \\
$(0,5)$ & $\{1,3,4\}$ & inside \\
$(1,4)$ & $\{2,3\}$ & inside \\
$(1,5),(2,5)$ & $\{3,4\}$ & inside \\
$(5,3),(5,4)$ & $\{0,2\}$ & outside \\
\end{tabular}
\end{appcerttable}

\begin{appcerttable}
\captionof{table}{$T_{3}$: ordered pairs with $\lvert I(a,d)\rvert\ge 2$}
\label{tab:I-T3}
\begin{tabular}{c|c|c}
$(a,d)$ & $I(a,d)$ & position \\ \hline
$(0,3)$ & $\{1,2\}$ & inside \\
$(0,4)$ & $\{1,2,3\}$ & inside \\
$(0,5)$ & $\{2,3,4\}$ & inside \\
$(1,4)$ & $\{2,3\}$ & inside \\
$(1,5)$ & $\{2,3,4\}$ & inside \\
$(2,5)$ & $\{3,4\}$ & inside \\
$(5,2),(5,3),(5,4)$ & $\{0,1\}$ & outside \\
\end{tabular}
\end{appcerttable}

\begin{appcerttable}
\captionof{table}{$T_{4}$: ordered pairs with $\lvert I(a,d)\rvert\ge 2$}
\label{tab:I-T4}
\begin{tabular}{c|c|c}
$(a,d)$ & $I(a,d)$ & position \\ \hline
$(0,3)$ & $\{1,2\}$ & inside \\
$(0,4)$ & $\{1,2,3\}$ & inside \\
$(0,5)$ & $\{2,4\}$ & inside \\
$(1,4)$ & $\{2,3\}$ & inside \\
$(1,5)$ & $\{2,4\}$ & inside \\
$(5,2),(5,3)$ & $\{0,1\}$ & outside \\
$(5,4)$ & $\{0,1,3\}$ & outside \\
\end{tabular}
\end{appcerttable}

\begin{appcerttable}
\captionof{table}{$T_{5}$: ordered pairs with $\lvert I(a,d)\rvert\ge 2$}
\label{tab:I-T5}
\begin{tabular}{c|c|c}
$(a,d)$ & $I(a,d)$ & position \\ \hline
$(0,3)$ & $\{1,2\}$ & inside \\
$(0,4)$ & $\{1,2,3\}$ & inside \\
$(0,5)$ & $\{3,4\}$ & inside \\
$(1,4)$ & $\{2,3\}$ & inside \\
$(1,5),(2,5)$ & $\{3,4\}$ & inside \\
$(5,2)$ & $\{0,1\}$ & outside \\
$(5,3),(5,4)$ & $\{0,1,2\}$ & outside \\
\end{tabular}
\end{appcerttable}

\begin{appcerttable}
\captionof{table}{$T_{6}$: ordered pairs with $\lvert I(a,d)\rvert\ge 2$}
\label{tab:I-T6}
\begin{tabular}{c|c|c}
$(a,d)$ & $I(a,d)$ & position \\ \hline
$(0,1)$ & $\{2,4\}$ & outside \\
$(0,3)$ & $\{1,2\}$ & inside \\
$(0,4)$ & $\{2,3\}$ & inside \\
$(0,5)$ & $\{2,3,4\}$ & inside \\
$(2,1)$ & $\{4,5\}$ & outside \\
$(2,5)$ & $\{3,4\}$ & inside \\
$(3,1)$ & $\{4,5\}$ & outside \\
$(5,3)$ & $\{0,1\}$ & outside \\
\end{tabular}
\end{appcerttable}

\begin{appcerttable}
\captionof{table}{$T_{7}$: ordered pairs with $\lvert I(a,d)\rvert\ge 2$}
\label{tab:I-T7}
\begin{tabular}{c|c|c}
$(a,d)$ & $I(a,d)$ & position \\ \hline
$(0,1)$ & $\{2,5\}$ & outside \\
$(0,3)$ & $\{1,2\}$ & inside \\
$(2,0)$ & $\{3,4\}$ & outside \\
$(2,1)$ & $\{4,5\}$ & outside \\
$(2,3)$ & $\{1,4\}$ & outside \\
$(2,5)$ & $\{3,4\}$ & inside \\
$(3,1),(4,1)$ & $\{0,5\}$ & outside \\
$(4,5)$ & $\{0,3\}$ & outside \\
\end{tabular}
\end{appcerttable}

\begin{appcerttable}
\captionof{table}{$T_{8}$: ordered pairs with $\lvert I(a,d)\rvert\ge 2$}
\label{tab:I-T8}
\begin{tabular}{c|c|c}
$(a,d)$ & $I(a,d)$ & position \\ \hline
$(0,3)$ & $\{1,2\}$ & inside \\
$(0,4)$ & $\{2,3\}$ & inside \\
$(0,5)$ & $\{1,2,3,4\}$ & inside \\
$(1,4),(1,5)$ & $\{2,3\}$ & inside \\
$(2,5)$ & $\{3,4\}$ & inside \\
\end{tabular}
\end{appcerttable}

\begin{appcerttable}
\captionof{table}{$T_{9}$: ordered pairs with $\lvert I(a,d)\rvert\ge 2$}
\label{tab:I-T9}
\begin{tabular}{c|c|c}
$(a,d)$ & $I(a,d)$ & position \\ \hline
$(0,3)$ & $\{1,2\}$ & inside \\
$(2,0)$ & $\{3,4,5\}$ & outside \\
$(2,1)$ & $\{4,5\}$ & outside \\
$(2,5)$ & $\{3,4\}$ & inside \\
$(3,0)$ & $\{4,5\}$ & outside \\
$(3,1)$ & $\{0,4,5\}$ & outside \\
$(4,1)$ & $\{0,5\}$ & outside \\
\end{tabular}
\end{appcerttable}

\begin{appcerttable}
\captionof{table}{$T_{10}$: ordered pairs with $\lvert I(a,d)\rvert\ge 2$}
\label{tab:I-T10}
\begin{tabular}{c|c|c}
$(a,d)$ & $I(a,d)$ & position \\ \hline
$(0,3)$ & $\{1,2\}$ & inside \\
$(2,0)$ & $\{3,4,5\}$ & outside \\
$(2,1)$ & $\{4,5\}$ & outside \\
$(2,3)$ & $\{1,5\}$ & outside \\
$(3,1)$ & $\{0,4\}$ & outside \\
$(4,1)$ & $\{0,5\}$ & outside \\
$(4,3)$ & $\{1,5\}$ & outside \\
\end{tabular}
\end{appcerttable}

\begin{appcerttable}
\captionof{table}{$T_{11}$: ordered pairs with $\lvert I(a,d)\rvert\ge 2$}
\label{tab:I-T11}
\begin{tabular}{c|c|c}
$(a,d)$ & $I(a,d)$ & position \\ \hline
$(0,3)$ & $\{1,2\}$ & inside \\
$(0,4)$ & $\{2,3\}$ & inside \\
$(2,0),(2,1)$ & $\{4,5\}$ & outside \\
$(2,3)$ & $\{1,5\}$ & outside \\
$(4,1)$ & $\{0,5\}$ & outside \\
$(4,3)$ & $\{0,1,5\}$ & outside \\
$(5,3)$ & $\{0,1\}$ & outside \\
\end{tabular}
\end{appcerttable}

\begin{appcerttable}
\captionof{table}{$T_{12}$: ordered pairs with $\lvert I(a,d)\rvert\ge 2$}
\label{tab:I-T12}
\begin{tabular}{c|c|c}
$(a,d)$ & $I(a,d)$ & position \\ \hline
$(0,3)$ & $\{1,2\}$ & inside \\
$(0,4)$ & $\{2,3\}$ & inside \\
$(0,5)$ & $\{1,3,4\}$ & inside \\
$(1,4)$ & $\{2,3\}$ & inside \\
$(2,5)$ & $\{3,4\}$ & inside \\
$(4,2)$ & $\{1,5\}$ & outside \\
$(5,3),(5,4)$ & $\{0,2\}$ & outside \\
\end{tabular}
\end{appcerttable}

\begin{appcerttable}
\captionof{table}{$T_{13}$: ordered pairs with $\lvert I(a,d)\rvert\ge 2$}
\label{tab:I-T13}
\begin{tabular}{c|c|c}
$(a,d)$ & $I(a,d)$ & position \\ \hline
$(0,1)$ & $\{2,4\}$ & outside \\
$(0,3)$ & $\{1,2\}$ & inside \\
$(0,5)$ & $\{2,4\}$ & inside \\
$(2,0)$ & $\{3,5\}$ & outside \\
$(2,1)$ & $\{4,5\}$ & outside \\
$(2,5)$ & $\{3,4\}$ & inside \\
$(3,1)$ & $\{0,4,5\}$ & outside \\
\end{tabular}
\end{appcerttable}

\begin{appcerttable}
\captionof{table}{$T_{14}$: ordered pairs with $\lvert I(a,d)\rvert\ge 2$}
\label{tab:I-T14}
\begin{tabular}{c|c|c}
$(a,d)$ & $I(a,d)$ & position \\ \hline
$(0,3)$ & $\{1,2\}$ & inside \\
$(0,4)$ & $\{2,3\}$ & inside \\
$(0,5)$ & $\{2,3,4\}$ & inside \\
$(1,4),(1,5)$ & $\{2,3\}$ & inside \\
$(2,1)$ & $\{4,5\}$ & outside \\
$(2,5)$ & $\{3,4\}$ & inside \\
$(3,1)$ & $\{4,5\}$ & outside \\
$(5,2),(5,3)$ & $\{0,1\}$ & outside \\
\end{tabular}
\end{appcerttable}

\begin{appcerttable}
\captionof{table}{$T_{15}$: ordered pairs with $\lvert I(a,d)\rvert\ge 2$}
\label{tab:I-T15}
\begin{tabular}{c|c|c}
$(a,d)$ & $I(a,d)$ & position \\ \hline
$(0,4),(0,5)$ & $\{2,3\}$ & inside \\
$(1,2)$ & $\{0,3\}$ & outside \\
$(1,4),(1,5)$ & $\{2,3\}$ & inside \\
$(2,0),(3,0)$ & $\{4,5\}$ & outside \\
$(3,4)$ & $\{2,5\}$ & outside \\
$(5,0)$ & $\{1,4\}$ & inside \\
$(5,2),(5,3)$ & $\{0,1\}$ & outside \\
\end{tabular}
\end{appcerttable}

\begin{appcerttable}
\captionof{table}{$T_{16}$: ordered pairs with $\lvert I(a,d)\rvert\ge 2$}
\label{tab:I-T16}
\begin{tabular}{c|c|c}
$(a,d)$ & $I(a,d)$ & position \\ \hline
$(0,3)$ & $\{1,2\}$ & inside \\
$(0,4)$ & $\{2,3\}$ & inside \\
$(0,5)$ & $\{2,4\}$ & inside \\
$(1,4)$ & $\{2,3\}$ & inside \\
$(2,1)$ & $\{4,5\}$ & outside \\
$(4,3)$ & $\{1,5\}$ & outside \\
$(5,2),(5,3)$ & $\{0,1\}$ & outside \\
$(5,4)$ & $\{0,3\}$ & outside \\
\end{tabular}
\end{appcerttable}

\begin{appcerttable}
\captionof{table}{$T_{17}$: ordered pairs with $\lvert I(a,d)\rvert\ge 2$}
\label{tab:I-T17}
\begin{tabular}{c|c|c}
$(a,d)$ & $I(a,d)$ & position \\ \hline
$(0,3)$ & $\{1,2\}$ & inside \\
$(0,4)$ & $\{2,3\}$ & inside \\
$(0,5)$ & $\{3,4\}$ & inside \\
$(1,4)$ & $\{2,3\}$ & inside \\
$(2,5)$ & $\{3,4\}$ & inside \\
$(3,1)$ & $\{4,5\}$ & outside \\
$(4,2)$ & $\{1,5\}$ & outside \\
$(5,2)$ & $\{0,1\}$ & outside \\
$(5,3)$ & $\{0,1,2\}$ & outside \\
$(5,4)$ & $\{0,2\}$ & outside \\
\end{tabular}
\end{appcerttable}

\begin{appcerttable}
\captionof{table}{$T_{18}$: ordered pairs with $\lvert I(a,d)\rvert\ge 2$}
\label{tab:I-T18}
\begin{tabular}{c|c|c}
$(a,d)$ & $I(a,d)$ & position \\ \hline
$(0,4)$ & $\{1,2,3\}$ & inside \\
$(0,5)$ & $\{2,3,4\}$ & inside \\
$(1,5)$ & $\{2,4\}$ & inside \\
$(2,1)$ & $\{3,5\}$ & outside \\
$(2,5)$ & $\{3,4\}$ & inside \\
$(5,2),(5,4)$ & $\{0,1\}$ & outside \\
\end{tabular}
\end{appcerttable}

\begin{appcerttable}
\captionof{table}{$T_{19}$: ordered pairs with $\lvert I(a,d)\rvert\ge 2$}
\label{tab:I-T19}
\begin{tabular}{c|c|c}
$(a,d)$ & $I(a,d)$ & position \\ \hline
$(0,4)$ & $\{1,2,3\}$ & inside \\
$(0,5),(2,5)$ & $\{3,4\}$ & inside \\
$(3,2)$ & $\{1,5\}$ & outside \\
$(5,2)$ & $\{0,1\}$ & outside \\
$(5,3)$ & $\{0,2\}$ & outside \\
$(5,4)$ & $\{0,1,2\}$ & outside \\
\end{tabular}
\end{appcerttable}

\begin{appcerttable}
\captionof{table}{$T_{20}$: ordered pairs with $\lvert I(a,d)\rvert\ge 2$}
\label{tab:I-T20}
\begin{tabular}{c|c|c}
$(a,d)$ & $I(a,d)$ & position \\ \hline
$(0,3)$ & $\{1,2\}$ & inside \\
$(0,4)$ & $\{1,2,3\}$ & inside \\
$(0,5),(1,4)$ & $\{2,3\}$ & inside \\
$(1,5)$ & $\{2,3,4\}$ & inside \\
$(2,0)$ & $\{4,5\}$ & outside \\
$(2,5)$ & $\{3,4\}$ & inside \\
$(3,0)$ & $\{4,5\}$ & outside \\
$(4,1)$ & $\{0,5\}$ & outside \\
$(5,2),(5,3)$ & $\{0,1\}$ & outside \\
\end{tabular}
\end{appcerttable}

\begin{appcerttable}
\captionof{table}{$T_{21}$: ordered pairs with $\lvert I(a,d)\rvert\ge 2$}
\label{tab:I-T21}
\begin{tabular}{c|c|c}
$(a,d)$ & $I(a,d)$ & position \\ \hline
$(0,4),(0,5)$ & $\{2,3\}$ & inside \\
$(2,0)$ & $\{4,5\}$ & outside \\
$(2,1)$ & $\{3,4,5\}$ & outside \\
$(2,5)$ & $\{3,4\}$ & inside \\
$(3,0),(3,1)$ & $\{4,5\}$ & outside \\
$(4,1)$ & $\{0,5\}$ & outside \\
$(4,2),(5,2)$ & $\{0,1\}$ & outside \\
\end{tabular}
\end{appcerttable}

\begin{appcerttable}
\captionof{table}{$T_{22}$: ordered pairs with $\lvert I(a,d)\rvert\ge 2$}
\label{tab:I-T22}
\begin{tabular}{c|c|c}
$(a,d)$ & $I(a,d)$ & position \\ \hline
$(0,3)$ & $\{1,2\}$ & inside \\
$(0,4)$ & $\{1,2,3\}$ & inside \\
$(1,4)$ & $\{2,3\}$ & inside \\
$(1,5)$ & $\{2,4\}$ & inside \\
$(2,0)$ & $\{4,5\}$ & outside \\
$(4,1),(4,3)$ & $\{0,5\}$ & outside \\
$(5,2),(5,3)$ & $\{0,1\}$ & outside \\
$(5,4)$ & $\{1,3\}$ & outside \\
\end{tabular}
\end{appcerttable}

\begin{appcerttable}
\captionof{table}{$T_{23}$: ordered pairs with $\lvert I(a,d)\rvert\ge 2$}
\label{tab:I-T23}
\begin{tabular}{c|c|c}
$(a,d)$ & $I(a,d)$ & position \\ \hline
$(0,3),(0,4)$ & $\{1,2\}$ & inside \\
$(0,5)$ & $\{2,4\}$ & inside \\
$(1,4)$ & $\{2,3\}$ & inside \\
$(1,5)$ & $\{2,4\}$ & inside \\
$(2,0)$ & $\{3,5\}$ & outside \\
$(5,2)$ & $\{0,1\}$ & outside \\
$(5,4)$ & $\{0,1,3\}$ & outside \\
\end{tabular}
\end{appcerttable}

\begin{appcerttable}
\captionof{table}{$T_{24}$: ordered pairs with $\lvert I(a,d)\rvert\ge 2$}
\label{tab:I-T24}
\begin{tabular}{c|c|c}
$(a,d)$ & $I(a,d)$ & position \\ \hline
$(0,3)$ & $\{1,2\}$ & inside \\
$(0,4)$ & $\{1,3\}$ & inside \\
$(0,5)$ & $\{2,3\}$ & inside \\
$(1,5)$ & $\{2,3,4\}$ & inside \\
$(3,0)$ & $\{4,5\}$ & outside \\
$(4,1)$ & $\{0,5\}$ & outside \\
$(4,3)$ & $\{0,2\}$ & outside \\
$(5,2),(5,3)$ & $\{0,1\}$ & outside \\
\end{tabular}
\end{appcerttable}

\begin{appcerttable}
\captionof{table}{$T_{25}$: ordered pairs with $\lvert I(a,d)\rvert\ge 2$}
\label{tab:I-T25}
\begin{tabular}{c|c|c}
$(a,d)$ & $I(a,d)$ & position \\ \hline
$(0,4)$ & $\{1,2,3\}$ & inside \\
$(0,5)$ & $\{2,3\}$ & inside \\
$(1,5)$ & $\{2,4\}$ & inside \\
$(2,0)$ & $\{4,5\}$ & outside \\
$(2,1)$ & $\{3,5\}$ & outside \\
$(2,5)$ & $\{3,4\}$ & inside \\
$(3,0)$ & $\{4,5\}$ & outside \\
$(4,1)$ & $\{0,5\}$ & outside \\
$(5,2)$ & $\{0,1\}$ & outside \\
\end{tabular}
\end{appcerttable}

\begin{appcerttable}
\captionof{table}{$T_{26}$: ordered pairs with $\lvert I(a,d)\rvert\ge 2$}
\label{tab:I-T26}
\begin{tabular}{c|c|c}
$(a,d)$ & $I(a,d)$ & position \\ \hline
$(0,3)$ & $\{1,2\}$ & inside \\
$(0,4),(0,5),(1,4),(1,5)$ & $\{2,3\}$ & inside \\
$(2,0),(2,1)$ & $\{4,5\}$ & outside \\
$(2,5)$ & $\{3,4\}$ & inside \\
$(3,0),(3,1)$ & $\{4,5\}$ & outside \\
$(4,1)$ & $\{0,5\}$ & outside \\
$(4,2),(4,3),(5,2),(5,3)$ & $\{0,1\}$ & outside \\
\end{tabular}
\end{appcerttable}

\begin{appcerttable}
\captionof{table}{$T_{27}$: ordered pairs with $\lvert I(a,d)\rvert\ge 2$}
\label{tab:I-T27}
\begin{tabular}{c|c|c}
$(a,d)$ & $I(a,d)$ & position \\ \hline
$(0,4)$ & $\{1,3\}$ & inside \\
$(0,5)$ & $\{2,3\}$ & inside \\
$(1,5)$ & $\{2,4\}$ & inside \\
$(2,1)$ & $\{3,5\}$ & outside \\
$(3,0)$ & $\{4,5\}$ & outside \\
$(3,2)$ & $\{1,4\}$ & outside \\
$(4,1)$ & $\{0,5\}$ & outside \\
$(4,3)$ & $\{0,2\}$ & outside \\
$(5,2)$ & $\{0,1\}$ & outside \\
\end{tabular}
\end{appcerttable}
\end{multicols}
\clearpage

\section{The seven-color certificate}\label{app:coloring}

\begin{multicols}{2}
\raggedcolumns
\setlength{\columnsep}{14pt}
\footnotesize

\par\noindent\begin{minipage}{\columnwidth}
\begin{Verbatim}[fontsize=\footnotesize,baselinestretch=1.06]
Tca=ac  (4x4)
0 1 2 0
1 2 0 1
2 0 1 2
0 1 2 0
\end{Verbatim}
\end{minipage}\par\addvspace{0.45\baselineskip}

\par\noindent\begin{minipage}{\columnwidth}
\begin{Verbatim}[fontsize=\footnotesize,baselinestretch=1.06]
Tcb=bc  (4x5)
0 1 2 0
3 2 0 3
4 0 3 4
5 3 1 5
0 1 2 0
\end{Verbatim}
\end{minipage}\par\addvspace{0.45\baselineskip}

\par\noindent\begin{minipage}{\columnwidth}
\begin{Verbatim}[fontsize=\footnotesize,baselinestretch=1.06]
Tda=ad  (5x4)
0 3 4 5 0
1 2 0 3 1
2 0 3 1 2
0 3 4 5 0
\end{Verbatim}
\end{minipage}\par\addvspace{0.45\baselineskip}

\par\noindent\begin{minipage}{\columnwidth}
\begin{Verbatim}[fontsize=\footnotesize,baselinestretch=1.06]
Tdb=bd  (5x5)
0 3 4 5 0
3 1 5 0 3
4 5 0 3 4
5 0 3 1 5
0 3 4 5 0
\end{Verbatim}
\end{minipage}\par\addvspace{0.45\baselineskip}

\par\noindent\begin{minipage}{\columnwidth}
\begin{Verbatim}[fontsize=\footnotesize,baselinestretch=1.06]
Tdca=acd  (8x4)
0 3 4 5 0 1 2 0
1 2 0 3 1 5 0 1
2 0 3 4 5 3 1 2
0 1 2 0 3 4 5 0
\end{Verbatim}
\end{minipage}\par\addvspace{0.45\baselineskip}

\par\noindent\begin{minipage}{\columnwidth}
\begin{Verbatim}[fontsize=\footnotesize,baselinestretch=1.06]
Tcba=abc  (4x8)
0 1 2 0
1 2 0 3
2 0 3 4
0 3 1 5
3 4 5 0
4 0 3 1
5 3 1 2
0 1 2 0
\end{Verbatim}
\end{minipage}\par\addvspace{0.45\baselineskip}

\par\noindent\begin{minipage}{\columnwidth}
\begin{Verbatim}[fontsize=\footnotesize,baselinestretch=1.06]
Tcda=adc  (8x4)
0 1 2 0 3 4 5 0
1 2 0 3 4 0 3 1
2 0 3 1 5 3 1 2
0 3 4 5 0 1 2 0
\end{Verbatim}
\end{minipage}\par\addvspace{0.45\baselineskip}

\par\noindent\begin{minipage}{\columnwidth}
\begin{Verbatim}[fontsize=\footnotesize,baselinestretch=1.06]
Tcab=bac  (4x8)
0 1 2 0
3 2 0 1
4 0 3 2
5 3 4 0
0 1 5 3
1 5 3 4
2 0 1 5
0 1 2 0
\end{Verbatim}
\end{minipage}\par\addvspace{0.45\baselineskip}

\par\noindent\begin{minipage}{\columnwidth}
\begin{Verbatim}[fontsize=\footnotesize,baselinestretch=1.06]
Tc^q a=ad^p  (13x4)
0 1 2 0 1 2 0 1 2 0 1 2 0
1 2 0 3 2 4 1 2 6 1 5 0 1
2 0 3 6 4 5 2 6 4 5 3 1 2
0 3 4 5 0 3 4 5 0 3 4 5 0
\end{Verbatim}
\end{minipage}\par\addvspace{0.45\baselineskip}

\par\noindent\begin{minipage}{\columnwidth}
\begin{Verbatim}[fontsize=\footnotesize,baselinestretch=1.06]
Td^p a=ac^q  (13x4)
0 3 4 5 0 3 4 5 0 3 4 5 0
1 2 0 3 6 1 5 0 3 1 5 0 1
2 0 1 2 4 5 2 6 1 2 0 1 2
0 1 2 0 1 2 0 1 2 0 1 2 0
\end{Verbatim}
\end{minipage}\par\addvspace{0.45\baselineskip}

\par\noindent\begin{minipage}{\columnwidth}
\begin{Verbatim}[fontsize=\footnotesize,baselinestretch=1.06]
Tc b^p=a^q c  (4x13)
0 1 2 0
1 2 0 3
2 0 3 4
0 3 6 5
1 2 4 0
2 4 5 3
0 1 2 4
1 2 6 5
2 6 4 0
0 1 5 3
1 5 3 4
2 0 1 5
0 1 2 0
\end{Verbatim}
\end{minipage}\par\addvspace{0.45\baselineskip}

\par\noindent\begin{minipage}{\columnwidth}
\begin{Verbatim}[fontsize=\footnotesize,baselinestretch=1.06]
Tc a^q=b^p c  (4x13)
0 1 2 0
3 2 0 1
4 0 1 2
5 3 2 0
0 6 4 1
3 1 5 2
4 5 2 0
5 0 6 1
0 3 1 2
3 4 5 0
4 0 3 1
5 3 1 2
0 1 2 0
\end{Verbatim}
\end{minipage}\par\addvspace{0.45\baselineskip}
\end{multicols}

\end{document}